\documentclass{article}
\usepackage{amssymb}
\usepackage{amsthm}
\usepackage{amsmath}
\usepackage{todonotes}
\usepackage{xcolor}
\usepackage{graphics}
\usepackage{subcaption}
\usepackage{geometry}

\newtheorem{theorem}{Theorem}
\newtheorem{lemma}{Lemma}

\theoremstyle{remark}
\newtheorem{remark}{Remark}
\newtheorem{example}{Example}

\newcommand{\diag}{\mathrm{diag}}
\newcommand{\ri}{\mathrm{i}}

\title{Non-Existence of Periodic Orbits for Forced-Damped Potential Systems in Bounded Domains}
\author{Florian Kogelbauer}

\begin{document}
	
\maketitle

\begin{abstract}
We prove $L^r$-estimates on periodic solutions of periodically-forced, linearly-damped mechanical systems with polynomially-bounded potentials. The estimates are applied to obtain a non-existence result of periodic solutions in bounded domains, depending on an upper bound on the gradient of the potential. The results are illustrated on examples.
\end{abstract}

\section{Introduction}

Periodic motions are of fundamental importance in the study of mechanical systems. Existence or non-existence of periodic orbit, depending on the forcing and damping of the system, provide a first insight in the structure of the phase space, as periodic orbits are - beside fixed-points - the simplest closed building blocks of the overall dynamics.\\
In conservative systems, a one-parameter family of periodic orbits around a fixed-point is guaranteed to exist under mild assumptions, cf. \cite{liapounoff1907probleme}. The periodic orbits are index by energy and form an invariant, two-dimensional sub-manifold in phase space. These manifolds allow for special coordinate systems, cf. \cite{kelley1967changes, kelley1968changes}, that are related to action-angle variables, cf. \cite{arnol2013mathematical}.\\
In forced-damped systems, isolated limit cycles, i.e., periodic orbits that appear as limit sets of close-by trajectories, typically exist under mild conditions. For small damping and forcing amplitude, existence can be proved by standard perturbation arguments, such as averaging, cf. \cite{guckenheimer2002nonlinear}. These techniques prove, to some extend, also basic qualitative properties of the periodic orbit - at least the proximity to a periodic orbit of the unforced system.  For general damping magnitude and forcing amplitude, topological techniques, including variants of topological degree theory, have been applied successfully to obtain existence of periodic orbits in mechanical systems, cf. \cite{gaines1977coincidence} and \cite{2019arXiv190703605B}. As the methods are based on continuous deformation, however, qualitative properties do not immediately follow in general.\\

From an application point of view, limit cycles can restrict the overall performance of some systems, such as machine tools \cite{insperger2004updated} or in aircraft design \cite{denegri2000limit}. There exists active control techniques to  limit the influence of limit cycles, cf. \cite{ko1997nonlinear}, as well as passive means that are based on a dynamical system approach, cf. \cite{habib2015suppression}.\\
It is therefore also necessary to ask the converse question: under which conditions can we guarantee that there is \textit{no} limit cycle (or even periodic orbit) that is entirely contained in a given domain (or that there is no piece of a periodic orbit in a given domain at all)?\\
In two-dimensional, autonomous systems, the absence of periodic orbits in a given domain can be tested easily with the Bendixson--Dulac criterion, cf. \cite{bendixson1901courbes, dulac1937recherche}. In particular, a purely damped system cannot generate any periodic orbits. There exists various extensions to higher-dimensional systems, e.g. based on vector-calculus operators, cf. \cite{busenberg1993method} and \cite{demidowitsch1966verallgemeinerung}. Also, for autonomous systems, extensions of the Bendixson--Dulac criterion to higher-dimensional invariant manifolds based on first integrals exist, cf. \cite{fevckan2001generalization}. We also refer to a generalization based on index theory, cf. \cite{smith1981index}. These generalization, however, do not generally apply to forced-damped mechanical systems, thus prompting our current interest in non-existence results and estimates for these equations.\\

In the present paper, we provide conditions that guarantee that a periodic orbit of a forced-damped mechanical system cannot be entirely contained in a given domain. The main idea is to obtain an a-priori upper bound on the periodic orbit $\mathbf{x}_p$ in some $L^p$-norm and then derive an implicit a-priori lower bound as well - both in terms of the external forcing. The key feature of the presented estimates is that they do not scale neutrally in the amplitude of the forcing, thus ensuring that - for sufficiently large amplitude - \textit{any} periodic orbit will leave the given domain eventually (at least for some time). This can be interpreted as a qualitative statement of the physical intuition that large forcing amplitude implies a large amplitude of the solution. All estimates are explicit and give precise bounds. Even though $L^p$-properties of periodic solution (and, in particular, limit cycles), have been useful in the study of mechanical systems, cf. \cite{liang2016rate}, it appears that $L^p$-estates have not been applied to obtain non-existence results for periodic orbits before.\\
We show that the critical amplitude, i.e., the amplitude of the forcing, at which any (potential) periodic leaves the domain of definition, depends differently on the linear stiffness, the linear damping coefficient and the magnitude of the nonlinearity. In our analysis, we consider both nonlinearly hardening as well as nonlinearly softening potentials (the analysis can easily be reversed for linearly soft systems). Our estimates are formulated for multiples of the forcing period to account for period-doubling bifurcations for large forcing amplitude.\\
We illustrate the estimates on a nonlinearly hardening and on a nonlinearly softening version of the two-dimensional, forced-damped Duffing oscillator.  

\subsection{Notation}
Let $I\subset\mathbb{R}^d$ be a bounded interval and let $\mathbf{f}:I\to\mathbb{R}, t\mapsto\mathbf{f}(t)$. For $1\leq p<\infty$ we define the $L^p$-norm of $\mathbf{f}$ as
\begin{equation}
\|\mathbf{f}\|_{L^p(I)}=\left(\int_I|\mathbf{f}(t)|^p\, dt\right)^{\frac{1}{p}},
\end{equation}
while for $p=\infty$, we set
\begin{equation}
\|\mathbf{f}\|_{L^\infty(I)}=\sup_{t\in I}|\mathbf{f}(t)|.
\end{equation}
For any function $\mathbf{f}:[0,T]\mapsto \mathbb{R}^d$, we write\begin{equation}\label{defmean}\hat{\mathbf{f}}=\frac{1}{T}\int_0^T\mathbf{f}(t)\, dt, \end{equation}for the mean of $\mathbf{f}$ and we write $\tilde{\mathbf{f}}:[0,T]\to\mathbb{R}^n$,\begin{equation}\label{defmeanfree}\tilde{\mathbf{f}}(t)=\mathbf{f}(t)-\hat{\mathbf{f}},\end{equation}for the mean-free part of $\mathbf{f}$.\\
For a number $r\in\mathbb{R}$ with $r\geq 1$, we write
\begin{equation}
r*=\frac{r}{r-1},
\end{equation}
the dual coefficient, to facilitate the notation.
For a vector $\mathbf{x}=(x_1,x_2,...,x_n)\in\mathbb{R}^d$, a \textit{subvector of} $\mathbf{x}$ is a vector $\underline{\mathbf{x}}\in\mathbb{R}^l$, $l\leq d$, such that $\underline{\mathbf{x}}=(x_{j_{1}},...,x_{j_l})$, for some indices $1\leq j_1,...,j_l\leq d$.

\section{Forced-Damped Linear Systems}
In this section, we recall some basic solvability properties of linear forced-damped systems, based on the expansion in Fourier series. Consider the linear second-order system
\begin{equation}\label{mainlinear}
\mathbf{M}\ddot{\mathbf{x}}(t)+\mathbf{C}\dot{\mathbf{x}}(t)+\mathbf{K}\mathbf{x}(t)=\mathbf{f}(t),
\end{equation}
for the dynamic variable $t\mapsto \mathbf{x}(t)\in\mathbb{R}^n$, with a symmetric, positive-definite mass matrix $\mathbf{M}$, a symmetric, positive definite stiffness matrix $\mathbf{K}$ and a symmetric, positive definite damping matrix $\mathbf{C}$. 
 We assume a (twice) continuously-differentiable, $T$-periodic external forcing
\begin{equation}
\mathbf{f}:\mathbb{R}\to\mathbb{R}^d,\quad \mathbf{f}(t+T)=\mathbf{f}(t),
\end{equation}
for all $t\in\mathbb{R}$.\\
Assuming the existence of a twice continuously differentiable, $T$-periodic orbit $t\mapsto\mathbf{x}_p(t)$, $\mathbf{x}_p(t+T)=\mathbf{x}_p(t)$, we can expand both $\mathbf{x}_p$ and $\mathbf{f}$ in Fourier series,
\begin{equation}
\mathbf{x}_p(t)=\sum_{n\in\mathbb{Z}}\hat{\mathbf{x}}_ne^{\mathbf{i}n\Omega t},\qquad \mathbf{f}(t)=\sum_{n\in\mathbb{Z}}\hat{\mathbf{f}}_ne^{\mathbf{i}n\Omega t},
\end{equation}
for the frequency $\Omega=\frac{2\pi}{T}$ and the Fourier coefficients
\begin{equation}
\hat{\mathbf{x}}_p=\frac{1}{T}\int_0^T\mathbf{x}_p(t) e^{-\mathbf{i}n\Omega t}\, dt,\qquad \hat{\mathbf{f}}=\frac{1}{T}\int_0^T\mathbf{f}(t) e^{-\mathbf{i}n\Omega t}\, dt.
\end{equation}
Passing to the frequency domain, equation \eqref{mainlinear} reads
\begin{equation}\label{linearFourier}
\sum_{n\in\mathbb{Z}}\Big(-(\Omega n)^2\mathbf{M}+\mathbf{i}\Omega n\mathbf{C}+\mathbf{K}\Big)\hat{\mathbf{x}}_n e^{\mathbf{i}n\Omega t}=\sum_{n\in\mathbb{Z}}\hat{\mathbf{f}}_ne^{\mathbf{i}n\Omega t}.
\end{equation}
In the case of vanishing damping, i.e., $\mathbf{C}\equiv 0$, equation \eqref{linearFourier} we can be solved for $\{\hat{\mathbf{f}}_n\}_{n\in\mathbb{Z}}$ uniquely, given any right-hand side $\{\hat{\mathbf{f}}_n\}_{n\in\mathbb{Z}}$, provided that 
\begin{equation}
\det\Big(\mathbf{K}-(\Omega n)^2\mathbf{M}\Big)\neq 0,	
\end{equation}
for all $n\in\mathbb{Z}$. If there exists a resonant wave-number for the matrices $\mathbf{M}$ and $\mathbf{K}$, i.e., an integer $n_0\in\mathbb{Z}$ such that $\det\Big(\mathbf{K}-(\Omega n_0)^2\mathbf{M}\Big)= 0$, equation \eqref{linearFourier} may still be solved for $\{\hat{\mathbf{x}}_n\}_{n\in\mathbb{Z}}$, provided that $\hat{\mathbf{f}}_{n_0}\in\text{range}\Big(\mathbf{K}-(\Omega n_0)^2\mathbf{M}\Big)$. \\
In the damped-forced case, i.e., for $\mathbf{C}\not\equiv 0$, assuming that the matrices $\mathbf{M}$, $\mathbf{K}$ and $\mathbf{C}$ diagonalize with the same set of eigenvectors,
\begin{equation}
\begin{split}
\mathbf{M}&=\mathbf{Q}^T\diag(m_1,...,m_d)\mathbf{Q},\\
\mathbf{K}&=\mathbf{Q}^T\diag(k_1,...,k_d)\mathbf{Q},\\
\mathbf{C}&=\mathbf{Q}^T\diag(c_1,...,c_d)\mathbf{Q},
\end{split}
\end{equation}
and writing
\begin{equation}
\mathbf{y}=\mathbf{Q}\mathbf{x},\qquad \mathbf{g}=\mathbf{Q}\mathbf{f},
\end{equation} 
we have that
\begin{equation}\label{Fouriercoef}
\hat{y}_n^j=\frac{1}{-(\Omega n)^2m_j+k_j+\ri \Omega n c_j}\hat{g}_n^j,
\end{equation}
where $\hat{\mathbf{y}}_n=(\hat{y}_n^1,...,\hat{y}_n^d)$ and $\hat{\mathbf{g}}_n=(\hat{g}_n^1,...,\hat{g}_n^d)$. Clearly, if $\mathbf{C}$ is (strictly) positive-definite, expression \eqref{Fouriercoef} is always well-defined, even if there are resonances in the undamped system. If there exists a $j$  such that $c_j=0$, then a resonance can occur and a solution only exists if the particular frequency is not exited, see the undamped case.\\
The amplitude of the solution given by \eqref{Fouriercoef} scales linearly with the amplitude of the forcing $\mathbf{f}$ and a sufficiently high amplitude will cause the solution to leave a bounded domain.\\
If the domain of validity is normalized to $\mathcal{V}=\{\mathbf{x}\in\mathbb{R}^d: |\mathbf{x}|<1\}$, it follows from Jensen's inequality,
\begin{equation}
\|\mathbf{x}\|_{L^\infty(0,T)}\geq \sqrt{T}\|\mathbf{x}\|_{L^2(0,T)},
\end{equation}
that any periodic orbit leaves the domain of validity provided that
\begin{equation}\label{nonexlinear}
A>\frac{1}{\sqrt{T}}\left(\sum_{n\in\mathbb{Z}}|\mathbf{Q}^T\diag\left(\frac{1}{-(\Omega n)^2m_1+k_1+\ri \Omega n c_1},...,\frac{1}{-(\Omega n)^2m_d+k_d+\ri \Omega n c_d}\right)\mathbf{Q}\hat{\mathbf{f}}_n|^2\right)^{-\frac{1}{2}}.
\end{equation}
We have chosen the $L^2$-norm as a lower bound for the $L^\infty$-norm in \eqref{nonexlinear}, since the $L^2$-norm can be written as a time-independent sum by Parseval's formula.

\section{Estimates on periodic orbits in Forced-Damped Nonlinear Potential Systems}
In this section, we give a criterion analogous to \eqref{nonexlinear} for nonlinear systems, both with nonlinearly hardening and nonlinearly softening potential. Consider the second-order potential system of the form
\begin{equation}\label{main}
\mathbf{M}\ddot{\mathbf{x}}(t)+\mathbf{C}\dot{\mathbf{x}}(t)+\mathbf{K}\mathbf{x}(t)=-\nabla U(\mathbf{x}(t))+\mathbf{f}(t),
\end{equation}
for the dynamic variable $\mathbf{x}(t)\in\mathbb{R}^d$, with a symmetric, positive semi-definite mass matrix $\mathbf{M}$, a symmetric, positive semi-definite stiffness matrix $\mathbf{K}$ and a symmetric, positive-definite damping matrix $\mathbf{C}$. We assume an at least continuously differentiable, $T$-periodic external forcing with Lipschitz-continuous derivative,
\begin{equation}
\mathbf{f}:\mathbb{R}\to\mathbb{R}^d,\quad \mathbf{f}(t+T)=\mathbf{f}(t),
\end{equation}
for all $t\in\mathbb{R}$.\\
Let $\mathcal{V}\subseteq\mathbb{R}^d$ be the domain of validity of equation \eqref{main} and assume that the potential $U:\mathcal{V}\to\mathbb{R}$ is twice continuously-differentiable. We call the potential $U$ (nonlinearly) \textit{hardening} if it satisfies the bounds
\begin{equation}\label{AssUhard}
\begin{split}
&\nabla U(\mathbf{x})\cdot\mathbf{x}\geq u_0|\mathbf{x}|^r,\\
&r>2,
\end{split}
\end{equation}
for  $u_0>0$ and all $\mathbf{x}\in \mathcal{V}$. We call the potential (nonlinearly) \textit{softening}, if it satisfies bounds of the form
\begin{equation}\label{AssUsoft}
\begin{split}
&\nabla U(\mathbf{x})\cdot\mathbf{x}\leq -u_0|\mathbf{x}|^r,\\
&r>2,
\end{split}
\end{equation}
for $u_0>0$ and all $\mathbf{x}\in \mathcal{V}$. In particular, any potential satisfying either \eqref{AssUhard} or \eqref{AssUsoft} does not contain any quadratic terms, i.e., all the linear contributions enter in \eqref{main} through $\mathbf{M}$, $\mathbf{C}$ and $\mathbf{K}$.\\

\begin{remark}
For radially-symmetric potentials $U(\mathbf{x})=u(\rho)$, with $\rho=|\mathbf{x}|$, conditions \eqref{AssUhard} and \eqref{AssUsoft} simplify to
\begin{equation}
\frac{\partial u}{\partial \rho}(\rho)\geq u_0\rho^{r-1},\quad \frac{\partial u}{\partial \rho}(\rho)\leq -u_0\rho^{r-1},
\end{equation}
for $r>2$, $u_1=0$ and all $\mathbf{x}\in \mathcal{V}$, i.e., \eqref{AssUhard} and \eqref{AssUsoft} impose polynomial growth conditions on the derivative of $u$.
\end{remark}

\begin{remark}
We did not include dependence on the spatial variable $\mathbf{x}$ in the external forcing. Assuming, however, appropriate bounds, one can prove similar results for $\mathbf{x}$-dependent forcings. Similarly, one can adopt the following estimates for time-dependent, periodic potentials. Since one would expect - at best - a quasi-periodic motion for a general time-dependence, we did not include these variations in the subsequent analysis.
\end{remark}

\begin{remark}
The following estimates also apply, if condition \eqref{AssUhard} is replaced by
\begin{equation}
\begin{split}
&\nabla U(\mathbf{x})\cdot\underline{\mathbf{x}}\geq u_0|\underline{\mathbf{x}}|^r-u_1,\\
&r>2,
\end{split}
\end{equation}
for $u_0,u_1>0$, where $\underline{\mathbf{x}}=\mathbf{P}\mathbf{x}$, for some projection matrix $\mathbf{P}$ with $\text{range}(\mathbf{P})=l<d$, provided that the matrices $\mathbf{M}$, $\mathbf{C}$ and $\mathbf{K}$ are still positive definite on $\text{range}(\mathbf{P})$. Similarly, condition \eqref{AssUsoft} can be weakened to
\begin{equation}
\begin{split}
&\nabla U(\mathbf{x})\cdot\underline{\mathbf{x}}\leq -u_0|\underline{\mathbf{x}}|^r+u_1,\\
&r>2,
\end{split}
\end{equation}
for $u_0,u_1>0$.
\end{remark}
From the definiteness of $\mathbf{M}$, $\mathbf{C}$ and $\mathbf{K}$ we deduce the bounds
\begin{equation}\label{eigMCK}
\begin{split}
&M_{min}|\mathbf{y}|\leq\mathbf{y}\cdot\mathbf{M}\mathbf{y}\leq M_{max}|\mathbf{y}|,\\
&C_{min}|\mathbf{y}|\leq\mathbf{y}\cdot\mathbf{C}\mathbf{y}\leq C_{max}|\mathbf{y}|,\\
&K_{min}|\mathbf{y}|\leq\mathbf{y}\cdot\mathbf{K}\mathbf{y}\leq K_{max}|\mathbf{y}|,
\end{split}
\end{equation}
for all $\mathbf{y}\in\mathbb{R}^d$, where $M_{min},K_{min}\geq 0$ and $C_{min}>0$.



\begin{lemma}\label{xdot}
Let $\mathbf{f}\in C^{Lip}(0,T)$ be non-constant and let $N\geq 1$ be an integer. The derivative of any $NT$-periodic, continuously differentiable solution $\mathbf{x}_p(t)$ to equation \eqref{main} can be bounded as
\begin{equation}\label{upperxdotL2}
\|\dot{\mathbf{x}}_{p}\|_{L^2(0,NT)}\leq \frac{\sqrt{N}}{C_{min}} \|\tilde{\mathbf{f}}\|_{L^2(0,T)},
\end{equation}
where $\tilde{\mathbf{f}}$ is the mean-free part of the forcing, cf. \eqref{defmeanfree}.
\end{lemma}

\begin{proof}
Taking the inner product in equation \eqref{main} with $\dot{\mathbf{x}}_p$ and integrating over $[0,NT]$, we obtain
\begin{equation}\label{multintxLr}
\int_0^{NT}\mathbf{M}\ddot{\mathbf{x}}_p(t)\cdot\dot{\mathbf{x}}_p(t)+\mathbf{C}\dot{\mathbf{x}}_p(t)\cdot\dot{\mathbf{x}}_p(t)+ \mathbf{K}\mathbf{x}_p(t)\cdot\dot{\mathbf{x}}_p(t)\, dt=\int_0^{NT}-\nabla U(\mathbf{x}_p(t))\cdot\dot{\mathbf{x}}_p(t)+\mathbf{f}(t)\cdot \dot{\mathbf{x}}_p(t)\, dt.
\end{equation}
From the $NT$-periodicity of $\mathbf{x}_p$ and $\dot{\mathbf{x}}_p$, as well as from the symmetry of $\mathbf{M}$ and $\mathbf{K}$, we deduce that (\ref{multint}) is equivalent to
\begin{equation}\label{intzeroxLr}
\int_0^{NT}\mathbf{f}(t)\cdot\dot{{\mathbf{x}}}_p\, dt=\int_0^{NT}\dot{\mathbf{x}}_p(t)\cdot\mathbf{C}\dot{\mathbf{x}}_p(t)\, dt,
\end{equation}
or, equivalently,
\begin{equation}
\int_0^{NT}\tilde{\mathbf{f}}(t)\cdot\dot{{\mathbf{x}}}_p\, dt=\int_0^{NT}\dot{\mathbf{x}}_p(t)\cdot\mathbf{C}\dot{\mathbf{x}}_p(t)\, dt,
\end{equation}
for the mean-free part of the forcing \eqref{defmeanfree}. Using the lower bound \eqref{eigMCK} for $\mathbf{C}$ and applying H\"older's inequality then proves \eqref{upperxdotL2}.
\end{proof}

The right-hand side in \eqref{upperxdotL2} only depends on the mean-free part of the forcing and the minimal damping coefficient. If the damping is large, the assumed periodic orbit changes - on average - more slowly.

\begin{lemma}\label{xLr}
Let $\mathbf{f}\in C^{Lip}(0,T)$, let $N\geq 1$ be an integer and let $\mathbf{x}_p$ be a continuously differentiable, $NT$-periodic solution to equation \eqref{main}.\\
If the potential $U$ is hardening, i.e., satisfies \eqref{AssUhard}, then the estimate 
\begin{equation}\label{upperxLr}
\|\mathbf{x}_p\|_{L^r(0,NT)}\leq N^\frac{1}{r}u_0^{-\frac{1}{r-1}}\left\|\mathbf{f}-\frac{M_{max}}{C_{min}}\dot{\mathbf{f}}\right\|_{L^{r*}(0,T)}^{\frac{1}{r-1}},
\end{equation}
holds. If the potential $U$ is softening, i.e., satisfies \eqref{AssUsoft}, then the estimate 
\begin{equation}\label{upperxLrsoft}
\|\mathbf{x}_p\|_{L^r(0,NT)}\leq N^{\frac{1}{r}}y^*,
\end{equation}
holds, where $y^*$ is the unique positive root of the polynomial
\begin{equation}\label{poly}
P(y)=u_0y^{r-1}-K_{max}T^{\frac{r-2}{r}}y-\left\|\mathbf{f}\right\|_{L^{r*}(0,T)}.
\end{equation}
\end{lemma}

\begin{proof}
	Taking the inner product in (\ref{main}) with $\mathbf{x}_p$ and integrating over $[0,NT]$ gives
	\begin{equation}\label{multint1xLr}
	\int_0^{NT}\mathbf{M}\ddot{\mathbf{x}}_p(t)\cdot\mathbf{x}_p(t)+\mathbf{C}\dot{\mathbf{x}}_p(t)\cdot\mathbf{x}_p(t)+ \mathbf{K}\mathbf{x}_p(t)\cdot\mathbf{x}_p(t)\, dt=\int_0^{NT}-\nabla U(\mathbf{x}_p(t))\cdot\mathbf{x}_p(t)+\mathbf{f}(t)\cdot \mathbf{x}_p(t)\, dt,
	\end{equation}
	which, after integration by parts, together with the symmetry of $\mathbf{M}$ and $\mathbf{C}$ as well as the $NT$-periodicity of $\mathbf{x}_p$, becomes
	\begin{equation}\label{multint2xLr}
	\begin{split}
	\int_0^{NT}\mathbf{K}\mathbf{x}_p(t)\cdot \mathbf{x}_p(t)+\nabla U(\mathbf{x}_p(t))\cdot\mathbf{x}_p(t)-\mathbf{f}(t)\cdot \mathbf{x}_p(t)\, dt&=\int_0^{NT}\dot{\mathbf{x}}_p(t)\cdot\mathbf{M}\dot{\mathbf{x}}_p(t)\, dt\\
	&\leq M_{max}\int_0^{NT}|\dot{\mathbf{x}}(t)|^2.
	\end{split}
	\end{equation}
	Invoking \eqref{upperxdotL2} to \eqref{multint2xLr}, we obtain
	\begin{equation}\label{multint3xLr}
	\int_0^{NT}\mathbf{K}\mathbf{x}_p(t)\cdot \mathbf{x}_p(t)+\nabla U(\mathbf{x}_p(t))\cdot\mathbf{x}_p(t)-\mathbf{f}(t)\cdot \mathbf{x}_p(t)\, dt\leq-\frac{M_{max}}{C_{min}}\int_0^{NT}\dot{\mathbf{f}}(t)\cdot{\mathbf{x}}_p\, dt.
	\end{equation}
	Assuming the lower bound \eqref{AssUhard}, it follows from the positive-definiteness of $\mathbf{K}$ and from equation \eqref{multint2xLr} together with H\"older's inequality that
	\begin{equation}
	u_0\|\mathbf{x}_p\|_{L^r(0,NT)}^r\leq \left\|\mathbf{f}-\frac{M_{max}}{C_{min}}\dot{\mathbf{f}}\right\|_{L^{r*}(0,NT)}\|\mathbf{x}_p\|_{L^r(0,NT)},
	\end{equation}
	from which \eqref{upperxLr} immediately follows.\\
	On the other hand, assuming the upper bound \eqref{AssUsoft}, equation \eqref{multint1xLr} together with the upper bound on $\mathbf{K}$ in \eqref{eigMCK} implies, after an integration by parts, that
	\begin{equation}\label{multint4Lxr}
	\int_0^{NT}-\mathbf{M}\dot{\mathbf{x}}_p(t)\cdot\dot{\mathbf{x}}_p(t)-u_0|\mathbf{x}_p(t)|^r+K_{max}|\mathbf{x}_p(t)|^2\, dt\geq\int_0^{NT}\mathbf{f}(t)\cdot\mathbf{x}_p(t)\, dt.
	\end{equation}
	Applying the Cauchy-Schwartz inequality and Jensen's inequality (which is possible thanks to the assumption $r>2$), the inequality \eqref{multint4Lxr} and the positive semi-definiteness of $\mathbf{M}$ implies that
	\begin{equation}
	u_0\|\mathbf{x}_p\|_{L^r(0,NT)}^r-K_{max}(NT)^{\frac{r-2}{r}}\|\mathbf{x}_p\|_{L^r(0,NT)}^2-\left\|\mathbf{f}\right\|_{L^{r*}(0,NT)}\|\mathbf{x}_p\|_{L^r(0,NT)}\leq 0.
	\end{equation}
	This is equivalent to 
	\begin{equation}
	\|\mathbf{x}_p\|_{L^r(0,NT)}\leq z^*,
	\end{equation}
	where $z^*$ is the unique positive root of the polynomial
	\begin{equation}
	\begin{split}
	\tilde{P}(z)&=u_0z^{r-1}-K_{max}(TN)^{\frac{r-2}{r}}z-\left\|\mathbf{f}\right\|_{L^{r*}(0,NT)}\\
	&=u_0z^{r-1}-K_{max}(TN)^{\frac{r-2}{r}}z-N^{\frac{1}{r*}}\left\|\mathbf{f}\right\|_{L^{r*}(0,T)}
	\end{split}
	\end{equation}
	(Since $\tilde{P}(0)<0$, $\tilde{P}'(0)<0$ and since $\tilde{P}'(z)$ only has one positive root, it follows that, indeed, $\tilde{P}(z)$ only has one positive root.)
	Rescaling $z=N^{\frac{1}{r}}y$ then gives  \eqref{upperxLrsoft}.
	
\end{proof}

\begin{remark}
For practical tests, the root $y^*$ can easily be calculated numerically, cf. Example \ref{Duffingsoftexpl}. For theoretical estimates building on the the estimate \eqref{upperxLr}, one might one to estimate the root $y^*$ from above by an explicit expression. In Lemma \ref{polybound} in the Appendix, we present such an estimate that approximates a polynomial of the form \eqref{poly} by a parabola with crest at the local minimum and gives an according upper bound on the unique positive root of \eqref{poly}.
\end{remark}

We note that the inequality \eqref{upperxLrsoft} is independent on the damping matrix $\mathbf{C}$.\\
The following lemma gives a stronger (in terms of the amplitude of the forcing) upper bound on the $L^2$-norm of $\dot{\mathbf{x}}$, using the estimates \eqref{upperxLr} and \eqref{upperxLrsoft}. Indeed, if the external forcing scales as $\mathbf{f}\mapsto A\mathbf{f}$, estimate \eqref{upperxdotL2Lr} gives an upper bound that scales as $A^{\frac{r}{2(r-1)}}$ instead of $A$ in \eqref{upperxdotL2}, remembering that $r\geq 2$ by assumption. Since $\frac{r}{2(r-1)}\geq 1$ by assumption, the amplitude of the solution can be bounded by the amplitude of the forcing through a nonlinear function, as compared to a linear relation for the case of $U\equiv 0$, cf. \eqref{Fouriercoef}.\\
The upper bounds \eqref{upperxLr} and \eqref{upperxLrsoft} are genuinely nonlinear, i.e., the only hold true for a non-vanishing nonlinear potential $U$. They become stronger for larger $u_0$ and become singular for $u_0\to 0$, i.e., in the linear regime (the root $y^*$ in \eqref{upperxLrsoft} does not even exist for $u_0=0$).

The following lemma gives an improved upper bound on the $L^2$-norm of the velocity of any periodic solution, based on the $L^r$-estimates obtained in Lemma \eqref{xLr}.

\begin{lemma}\label{xdotxLr}
	Let $\mathbf{f}\in C^{Lip}(0,T)$ be non-constant and let $N\geq 1$ be an integer. It the potential is nonlinearly hardening, i.e., satisfies \eqref{AssUhard}, the derivative of any $NT$-periodic, continuously differentiable solution $\mathbf{x}_p(t)$ to equation \eqref{main} can be bounded as
	\begin{equation}\label{upperxdotL2Lr}
	\|\dot{\mathbf{x}}_{p}\|_{L^2(0,NT)}\leq u_0^{-\frac{1}{2(r-1)}}\sqrt{\frac{N}{C_{min}}} \|\dot{\mathbf{f}}\|_{L^{r*}(0,T)}^{\frac{1}{2}}\left\|\mathbf{f}-\frac{M_{max}}{C_{min}}\dot{\mathbf{f}}\right\|_{L^{r*}(0,T)}^{\frac{1}{2(r-1)}}.
	\end{equation}
	If the potential is nonlinearly softening, i.e., satisfies \eqref{AssUsoft}, the derivative of any $NT$-periodic, continuously differentiable solution $\mathbf{x}_p(t)$ to equation \eqref{main} can be bounded as
		\begin{equation}\label{upperxdotL2Lrsoft}
	\|\dot{\mathbf{x}}_{p}\|_{L^2(0,NT)}\leq\frac{N^{\frac{1}{2r*}}}{\sqrt{C_{min}}} \|\dot{\mathbf{f}}\|_{L^{r*}(0,T)}^{\frac{1}{2}}\sqrt{y^*},
	\end{equation}
	where again $y^*$ is the unique positive root of the polynomial
	\begin{equation}
	P(y)=u_0y^{r-1}-K_{max}(NT)^{\frac{r-2}{r}}y-\left\|\mathbf{f}\right\|_{L^{r*}(0,NT)}.
	\end{equation}
\end{lemma}

\begin{proof}
	Taking the inner product in equation \eqref{main} with $\dot{\mathbf{x}}_p$ and integrating over $[0,NT]$, we obtain
	\begin{equation}\label{multintxdot}
	\int_0^{NT}\mathbf{M}\ddot{\mathbf{x}}_p(t)\cdot\dot{\mathbf{x}}_p(t)+\mathbf{C}\dot{\mathbf{x}}_p(t)\cdot\dot{\mathbf{x}}_p(t)+ \mathbf{K}\mathbf{x}_p(t)\cdot\dot{\mathbf{x}}_p(t)\, dt=\int_0^{NT}-\nabla U(\mathbf{x}_p(t))\cdot\dot{\mathbf{x}}_p(t)+\mathbf{f}(t)\cdot \dot{\mathbf{x}}_p(t)\, dt.
	\end{equation}
	From the $NT$-periodicity of $\mathbf{x}_p$ and $\dot{\mathbf{x}}_p$, as well as from the symmetry of $\mathbf{M}$ and $\mathbf{K}$, we deduce that (\ref{multint}) is equivalent to
	\begin{equation}\label{intzeroxdot}
	\int_0^{NT}\mathbf{f}(t)\cdot\dot{{\mathbf{x}}}_p\, dt=\int_0^{NT}\dot{\mathbf{x}}_p(t)\cdot\mathbf{C}\dot{\mathbf{x}}_p(t)\, dt,
	\end{equation}
	or, equivalently,
	\begin{equation}
	-\int_0^{NT}\dot{\mathbf{f}}(t)\cdot{\mathbf{x}}_p\, dt=\int_0^{NT}\dot{\mathbf{x}}_p(t)\cdot\mathbf{C}\dot{\mathbf{x}}_p(t)\, dt.
	\end{equation}
	Using the lower bound \eqref{eigMCK} for $\mathbf{C}$ and applying H\"older's inequality implies that
	\begin{equation}
	C_{min}\|\dot{\mathbf{x}}_p\|_{L^2(0,NT)}^2\leq \|\dot{\mathbf{f}}\|_{L^{r*}(0,NT)}\|\mathbf{x}_p\|_{L^r(0,NT)}.
	\end{equation}
	The claims now follow from Lemma \eqref{xLr} together with the scaling in $N$ of the right-hand side.
	\end{proof}

\section{Non-existence of periodic orbits in bounded domains}
In this section, we present some conditions on the non-existence of periodic orbits, assuming that the gradient of the potential satisfies a certain upper bound in the domain of validity.\\
Due to the assumed polynomial bounds on the potential and thanks to the estimates \eqref{upperxLr} and \eqref{upperxLr}, these conditions do not scale neutrally in the amplitude of the forcing. This implies that, for an external forcing with sufficiently large amplitude, any (potential) periodic orbit will necessarily leave the domain of validity, or, to phrase it differently: We obtain a lower bound on the maximum of a periodic orbit in terms of the amplitude of the external forcing.\\

\begin{theorem}
	Assume that the potential $U$ satisfies the lower bound \eqref{AssUhard} and assume that the gradient of the potential $U$ can be bounded as
	\begin{equation}\label{boundgradU}
	|\nabla U(\mathbf{x})|\leq U_0,
	\end{equation}
	for some $U_0>0$, all $\mathbf{x}\in \mathcal{V}$. If the $T$-periodic forcing $\mathbf{f}:[0,T]\to\mathbb{R}^d$ satisfies the bound
	\begin{equation}\label{nonex}
	\|\mathbf{f}\|_{L^2(0,T)}^2>U_0\|\mathbf{f}\|_{L^1(0,T)}+u_0^{-\frac{1}{r-1}}\|\mathbf{M}\ddot{\mathbf{f}}-\mathbf{C}\dot{\mathbf{f}}+\mathbf{K}\mathbf{f}\|_{L^{r*}(0,T)}\left\|\mathbf{f}-\frac{M_{max}}{C_{min}}\dot{\mathbf{f}}\right\|_{L^{r*}(0,T)}^{\frac{1}{r-1}},
	\end{equation}
	then there does not exist an $NT$-periodic orbit entirely contained in $\mathcal{V}$ for any integer $N\geq 1$.
\end{theorem}

\begin{proof}
	Assume, to the contrary, that there exists an $NT$-periodic orbit $t\mapsto \mathbf{x}_p(t)$, $\mathbf{x}(t+NT)=\mathbf{x}(t)$. Multiplying equation \eqref{main} with $\mathbf{f}$ and integrating over $[0,NT]$, we obtain
	\begin{equation}
	\int_0^{NT} \mathbf{M}\ddot{\mathbf{x}}_p(t)\cdot\mathbf{f}(t)+\mathbf{C}\dot{\mathbf{x}}_p(t)\cdot\mathbf{f}(t)+\mathbf{K}\mathbf{x}_p(t)\cdot\mathbf{f}(t)\, dt=\int_0^{NT}-\nabla U(\mathbf{x}_p(t))\cdot\mathbf{f}(t)+|\mathbf{f}(t)|^2\, dt,
	\end{equation}
	which, after an integration by parts together with symmetry of $\mathbf{M}$, $\mathbf{C}$ and $\mathbf{K}$ and with assumption \eqref{boundgradU}, implies that
	\begin{equation}\label{boundgrad1}
	\int_0^{NT}|\mathbf{f}(t)|^2\, dt\leq\int_0^{NT} U_0 |\mathbf{f}(t)|+|\mathbf{M}\ddot{\mathbf{f}}(t)-\mathbf{C}\dot{\mathbf{f}}(t)+\mathbf{K}\mathbf{f}(t)||\mathbf{x}_p(t)|\, dt.
	\end{equation}
	Applying H\"older's inequality to equation \eqref{boundgrad1} and using the estimate \eqref{upperxLr} gives
	\begin{equation}\label{boundgrad2}
	\begin{split}
	\|\mathbf{f}\|_{L^2(0,NT)}^2&\leq U_0\|\mathbf{f}\|_{L^1(0,NT)}+\|\mathbf{M}\ddot{\mathbf{f}}-\mathbf{C}\dot{\mathbf{f}}+\mathbf{K}\mathbf{f}\|_{L^{r*}(0,NT)}\|\mathbf{x}_p\|_{L^r(0,NT)}\\
	&\leq U_0\|\mathbf{f}\|_{L^1(0,NT)}+u_0^{-\frac{1}{r-1}}\|\mathbf{M}\ddot{\mathbf{f}}-\mathbf{C}\dot{\mathbf{f}}+\mathbf{K}\mathbf{f}\|_{L^{r*}(0,NT)}\left\|\mathbf{f}-\frac{M_{max}}{C_{min}}\dot{\mathbf{f}}\right\|_{L^{r*}(0,NT)}^{\frac{1}{r-1}}.
	\end{split}
	\end{equation}
	Since the left-hand side and the right-hand side of \eqref{boundgrad2} both scale with $N$, we arrive at a contradiction to assumption \eqref{nonex}. This proves the claim.
\end{proof}

\begin{theorem}\label{nonexthm}
	Assume that the potential $U$ satisfies the upper bound \eqref{AssUsoft} and assume that the gradient of the potential $U$ can be bounded as
	\begin{equation}\label{boundgradUsoft}
	|\nabla U(\mathbf{x})|\leq U_0,
	\end{equation}
	for some $U_0>0$, all $\mathbf{x}\in \mathcal{V}$. If the $T$-periodic forcing $\mathbf{f}:[0,T]\to\mathbb{R}^d$ satisfies the bound
	\begin{equation}\label{nonexsoft}
	\|\mathbf{f}\|_{L^2(0,T)}^2>U_0\|\mathbf{f}\|_{L^1(0,T)}+\|\mathbf{M}\ddot{\mathbf{f}}-\mathbf{C}\dot{\mathbf{f}}+\mathbf{K}\mathbf{f}\|_{L^{r*}(0,T)}y^*,
	\end{equation}
	where $y^*$ is the unique positive root of the polynomial
	\begin{equation}
	P(y)=u_0y^{r-1}-K_{max}T^{\frac{r-2}{r}}y-\left\|\mathbf{f}\right\|_{L^{r*}(0,T)},
	\end{equation}
	then there does not exist an $NT$-periodic orbit entirely contained in $\mathcal{V}$.
\end{theorem}

\begin{proof}
	Assume, to the contrary, that there exists an $NT$-periodic orbit $t\mapsto \mathbf{x}_p(t)$, $\mathbf{x}(t+NT)=\mathbf{x}(t)$. Multiplying equation \eqref{main} with $\mathbf{f}$ and integrating over $[0,NT]$, we obtain
	\begin{equation}
	\int_0^{NT} \mathbf{M}\ddot{\mathbf{x}}_p(t)\cdot\mathbf{f}(t)+\mathbf{C}\dot{\mathbf{x}}_p(t)\cdot\mathbf{f}(t)+\mathbf{K}\mathbf{x}_p(t)\cdot\mathbf{f}(t)\, dt=\int_0^{NT}-\nabla U(\mathbf{x}_p(t))\cdot\mathbf{f}(t)+|\mathbf{f}(t)|^2\, dt,
	\end{equation}
	which, after an integration by parts together with symmetry of $\mathbf{M}$, $\mathbf{C}$ and $\mathbf{K}$ and with assumption \eqref{boundgradUsoft}, implies that
	\begin{equation}\label{boundgrad1soft}
	\int_0^{NT}|\mathbf{f}(t)|^2\, dt\leq\int_0^{NT} U_0 |\mathbf{f}(t)|+|\mathbf{M}\ddot{\mathbf{f}}(t)-\mathbf{C}\dot{\mathbf{f}}(t)+\mathbf{K}\mathbf{f}(t)||\mathbf{x}_p(t)|\, dt.
	\end{equation}
	Applying H\"older's inequality to equation \eqref{boundgrad1soft} and using the estimate \eqref{upperxLrsoft} gives
	\begin{equation}\label{boundgrad2soft}
	\begin{split}
	\|\mathbf{f}\|_{L^2(0,NT)}^2&\leq U_0\|\mathbf{f}\|_{L^1(0,NT)}+\|\mathbf{M}\ddot{\mathbf{f}}-\mathbf{C}\dot{\mathbf{f}}+\mathbf{K}\mathbf{f}\|_{L^{r*}(0,NT)}\|\mathbf{x}_p\|_{L^r(0,NT)}\\
	&\leq U_0\|\mathbf{f}\|_{L^1(0,NT)}+\|\mathbf{M}\ddot{\mathbf{f}}-\mathbf{C}\dot{\mathbf{f}}+\mathbf{K}\mathbf{f}\|_{L^{r*}(0,NT)}N^{\frac{1}{r}}y^*,
	\end{split}
	\end{equation}
	where $y^*$ is the unique positive root of the polynomial
	\begin{equation}
	P(y)=u_0y^{r-1}-K_{max}T^{\frac{r-2}{r}}y-\left\|\mathbf{f}\right\|_{L^{r*}(0,T)}.
	\end{equation}
	Since both sides in \eqref{boundgrad2soft} scale neutrally in $N$, we obtain a contradiction to assumption  \eqref{nonexsoft}, which proves the claim.
\end{proof}

\begin{remark}
Clearly, condition \eqref{AssUhard} and condition \eqref{boundgradU} cannot hold at the same time globally. Mechanical models, however, are generally only justified on some bounded domain around the position of rest, as higher displacements would violate the validity of the model. If condition \eqref{nonex} is satisfied on a domain on which the gradient of $U$ can be bounded as in \eqref{boundgradU}, we can infer, for large enough amplitude of the external forcing, that the validity of the model breaks down, as the periodic orbit is leaving the domain of validity.
\end{remark}

\begin{remark}
	From the upper bound \eqref{upperxdotL2} it immediately follows that
	\begin{equation}\label{Linfty}
	\|\mathbf{x}_p\|_{L^\infty(0,NT)}\leq |\mathbf{x}_0|+\frac{N\sqrt{T}}{C_{min}}\|\tilde{\mathbf{f}}\|_{L^2(0,T)},
	\end{equation}
	where $\mathbf{x}_0=\mathbf{x}_p(0)$. If we chose to use the estimate \eqref{Linfty} instead of, say, the upper bound \eqref{upperxLr} in the proof of Theorem \eqref{nonexthm}, one would obtain a non-existence criterion of the form
	\begin{equation}
	\|\mathbf{f}\|_{L^2(0,T)}^2>U_0\|\mathbf{f}\|_{L^1(0,T)}+\|\mathbf{M}\ddot{\mathbf{f}}-\mathbf{C}\dot{\mathbf{f}}+\mathbf{K}\mathbf{f}\|_{L^{1}(0,T)}\left(|\mathbf{x}_0|+\frac{N\sqrt{T}}{C_{min}}\|\tilde{\mathbf{f}}\|_{L^2(0,T)}\right),
	\end{equation} 
	which has a quadratic amplitude term on the right-hand side as well as on the left-hand side. Therefore, it is not guaranteed that \textit{any} periodic orbit will leave the domain of definition (where he gradient of the potential can be bounded accordingly). This shows that the $L^p$-estimates \eqref{upperxLr} and \eqref{upperxLrsoft} are crucial for our line of reasoning.
\end{remark}

\section{Examples}

\begin{example}\label{DuffMean}
	Consider the forced-damped Duffing oscillator with a hardening cubic stiffness,
	\begin{equation}\label{Duffinghard}
	\ddot{x}+c\dot{x}+kx=-\delta x^3+A\sin(n\omega t),
	\end{equation}
	for $c>0$ the linear damping coefficient, $k>0$ the linear stiffness coefficient, $\delta>0$ the cubic stiffness coefficient, $A>0$ the forcing magnitude, $\omega>0$ the fundamental forcing frequency and $n\in\mathbb{N}$ an oscillation parameter.
	The nonlinear potential associated to equation \eqref{Duffinghard} is given by
	\begin{equation}
	U(x)=\delta\frac{x^4}{4},
	\end{equation}
	which is depicted in Figure \ref{x^4}. We can choose the constants in \eqref{AssUhard} as 
	\begin{equation}\label{UDuffing}
	u_0=\delta, \quad r=4.
	\end{equation}
	As our domain of validity, we choose the unit interval $I=[0,1]$, where we can bound the derivative of the potential $U$ as 
	\begin{equation}
	|U'(x)|=\delta|x|^3\leq \delta.
	\end{equation}
	This implies that we can choose $U_0=\delta$. Consequently, condition \eqref{nonex} takes the form
	\begin{equation}\label{condDuff}
	\begin{split}
			\frac{\pi}{\omega}A^2>&\delta\|A\sin(n\omega t)\|_{L^1(0,T)}+\delta^{-\frac{1}{3}}\|A(k-n^2\omega^2)\sin(n\omega t)-Acn\omega\cos(n \omega t)\|_{L^{\frac{4}{3}}(0,T)}\\
			&\qquad\qquad\qquad\times\left\|A\sin(n \omega t)-A\frac{n\omega }{c}\cos(n\omega t)\right\|_{L^{\frac{4}{3}}(0,T)}^{\frac{1}{3}}.
		\end{split}	
	\end{equation}
Figure \ref{F} shows the behavior of the function
\begin{footnotesize}
\begin{equation}\label{defF}
	F(A):=\frac{\pi}{\omega}A^2-\delta\|\sin(n\omega t)\|_{L^1(0,T)}A-\delta^{-\frac{1}{3}}\|(k-n^2\omega^2)\sin(n\omega t)-cn\omega\cos(n \omega t)\|_{L^{\frac{4}{3}}(0,T)}\left\|\sin(n \omega t)-\frac{n\omega }{c}\cos(n\omega t)\right\|_{L^{\frac{4}{3}}(0,T)}^{\frac{1}{3}}A^{\frac{4}{3}},
\end{equation}
\end{footnotesize}
for different parameter values. A sign-change of $F$ implies that any periodic orbit leaves the domain of validity.\\
Figure \ref{Astar} shows the behavior of the unique positive root of the function \eqref{defF} in dependence on the damping $c$, the magnitude of the nonlinear potential $\delta$ and the linear stiffness $k$. All three plots indicate a large forcing amplitude for small parameter values, as $c$ and $k$ appear in the denominator of the solution for linear systems, cf. \eqref{linearFourier}, while small values of $\delta$ indicate weak nonlinearity. As the nonlinearp potential is hardening, higher values of $\delta$, i.e., stronger influence of the potential, require higher amplitudes for a periodic orbit to leave the domain of validity.

	\begin{figure}
		\begin{subfigure}{.5\textwidth}
			\centering
			\includegraphics[width=.8\linewidth]{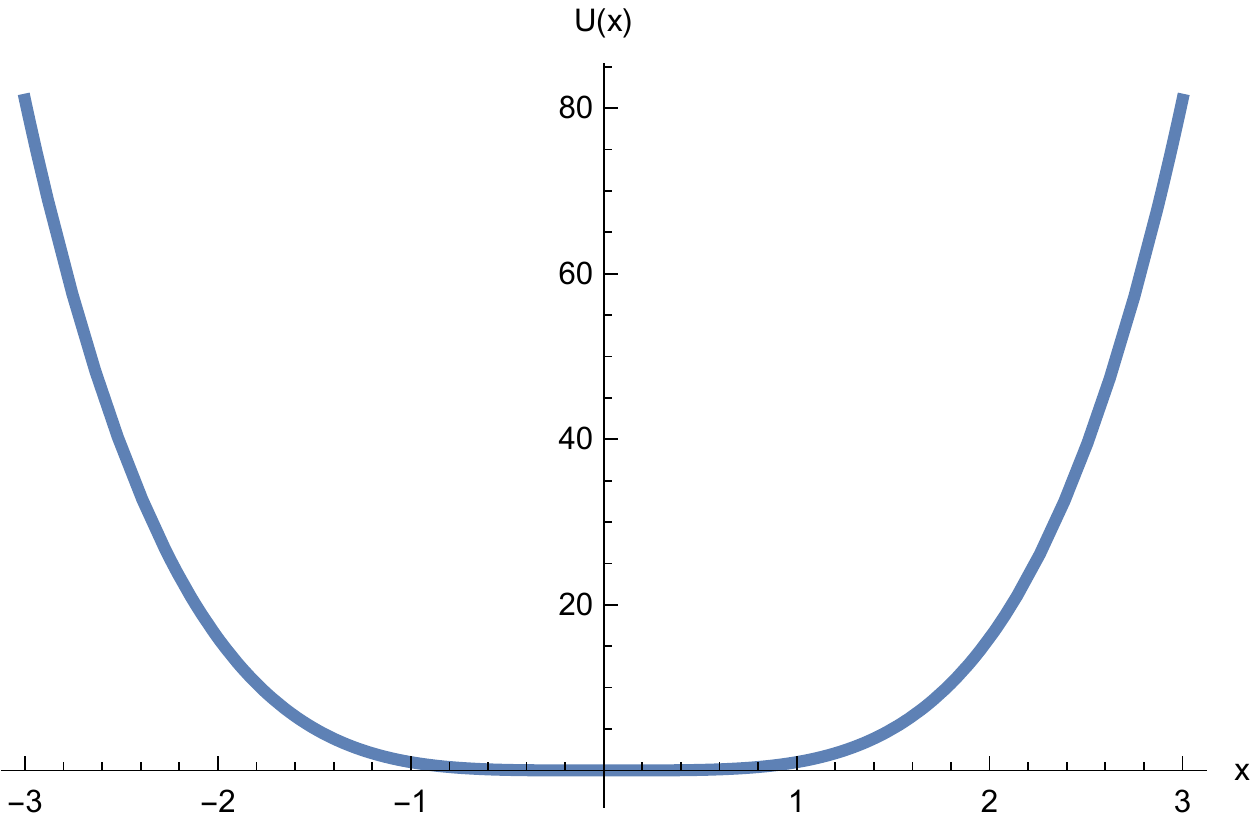}
			\caption{The nonlinear potential $U(x)=x^4$.}
			\label{x^4}
		\end{subfigure}%
		\begin{subfigure}{.5\textwidth}
			\centering
			\includegraphics[width=.7\linewidth]{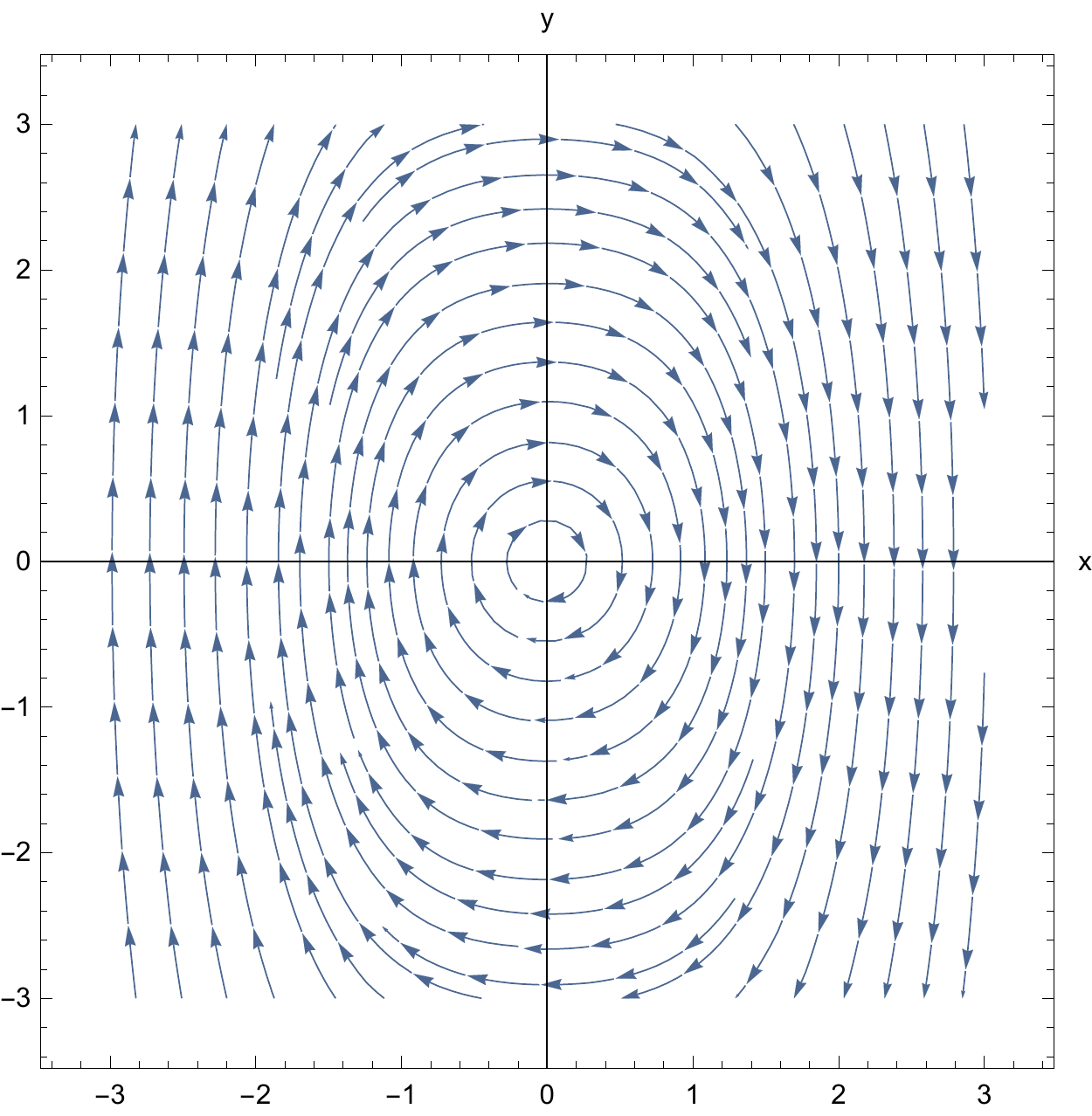}
			\caption{Dynamics of system \eqref{Duffinghard} around the origin for $c=A=0$.}
			\label{Dyn_Duff}
		\end{subfigure}
		\caption{The potential $U$ and the phase portrait of the unforced and undamped hardening Duffing oscillator.}
		\label{}
	\end{figure}

	\begin{figure}
	\begin{subfigure}{.5\textwidth}
		\centering
		\includegraphics[width=.8\linewidth]{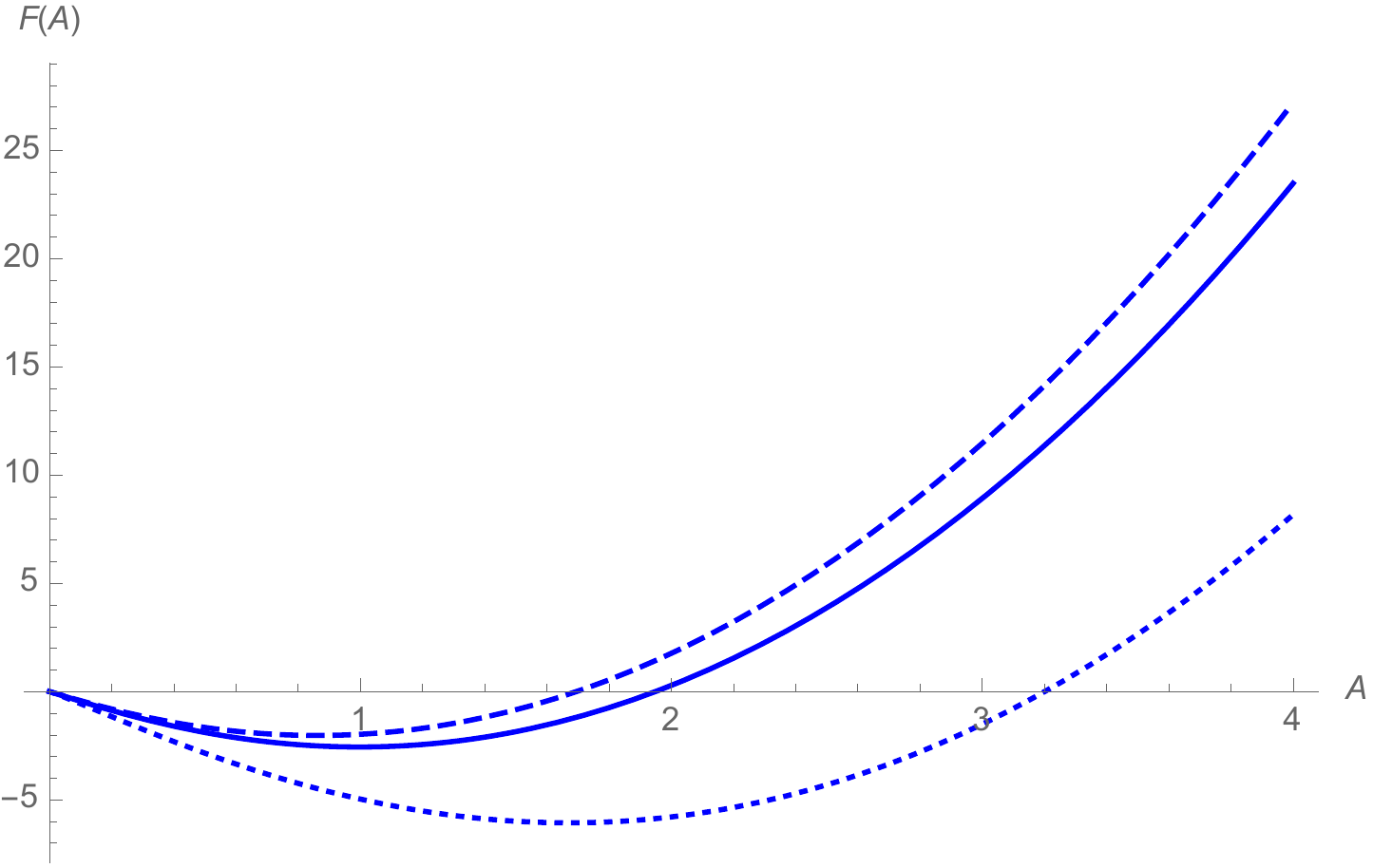}
		\caption{The function \eqref{defF} for $\omega = 1, k = 1.1, \delta = 1, n = 1$ and $c=0.01$ (solid), $c=0.1$ (dashed) and $c=1$ (dotted).}
		\label{Fc}
	\end{subfigure}%
	\begin{subfigure}{.5\textwidth}
		\centering
		\includegraphics[width=.7\linewidth]{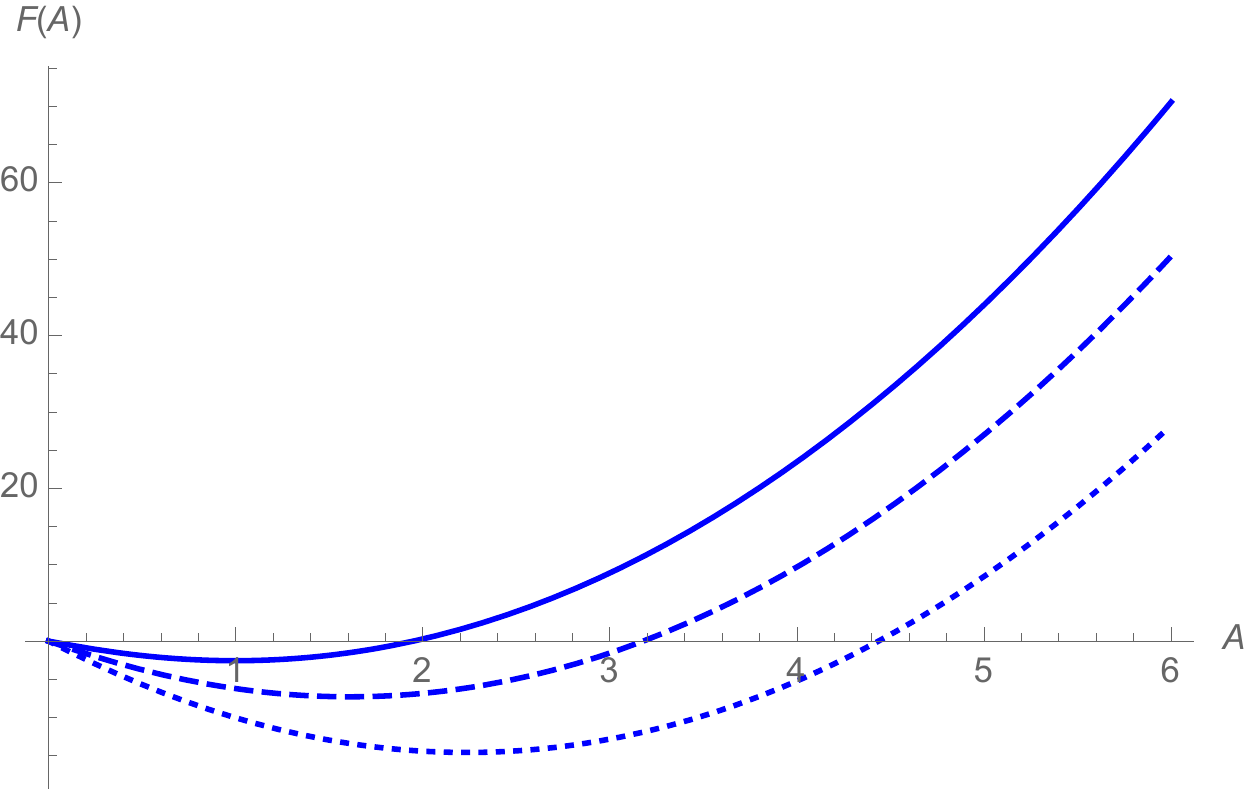}
		\caption{The function \eqref{defF} for $\omega = 1, k = 1.1, c=0.1, n = 1$ and $\delta=1$ (solid), $\delta=2$ (dashed) and $\delta=3$ (dotted).}
		\label{Fd}
	\end{subfigure}
	\centering
	\begin{subfigure}{.5\textwidth}
		\centering
		\includegraphics[width=.7\linewidth]{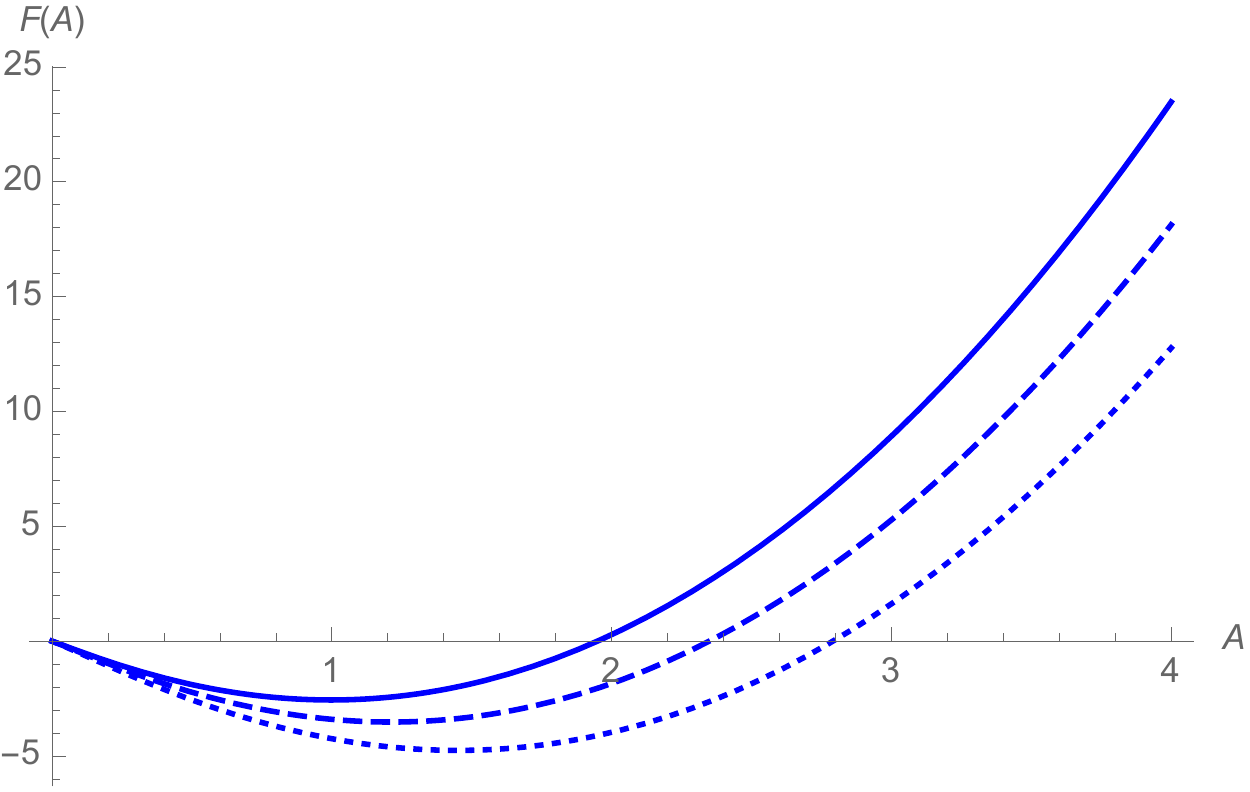}
		\caption{The function \eqref{defF} for $\omega = 1, c=0.1, \delta = 1, n = 1$ and $k=1.1$ (solid), $k=1.15$ (dashed) and $k=1.2$ (dotted).}
		\label{Fk}
	\end{subfigure}
	\caption{The behavior of the function \eqref{defF} for different parameter values of $c$, $\delta$ and $k$.}
	\label{F}
\end{figure}

	\begin{figure}
	\begin{subfigure}{.5\textwidth}
		\centering
		\includegraphics[width=.8\linewidth]{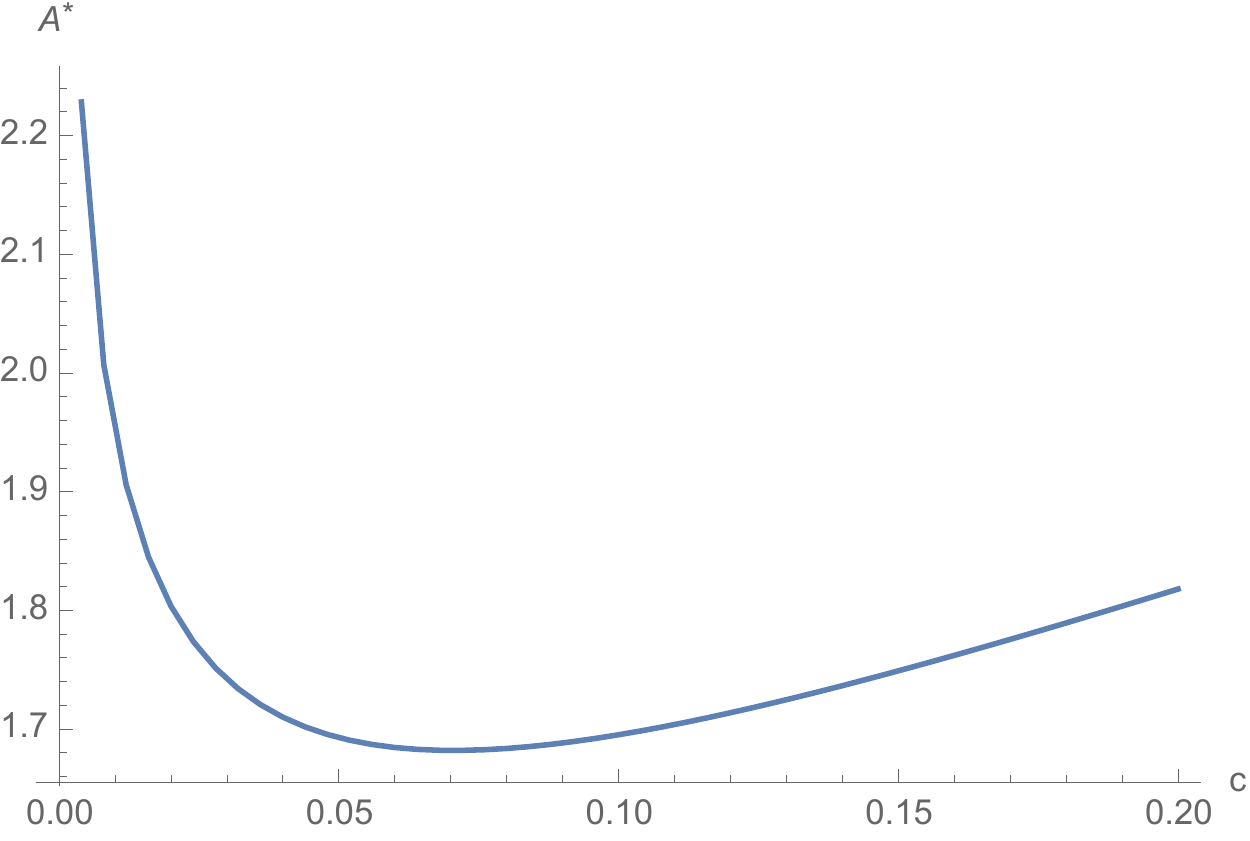}
		\caption{The unique positive root of the function \eqref{defF} for $\omega = 1, k = 1.1, \delta = 1, n = 1$ and $c\in \{0,0.2\}$}
		\label{Astarchard}
	\end{subfigure}%
	\begin{subfigure}{.5\textwidth}
		\centering
		\includegraphics[width=.7\linewidth]{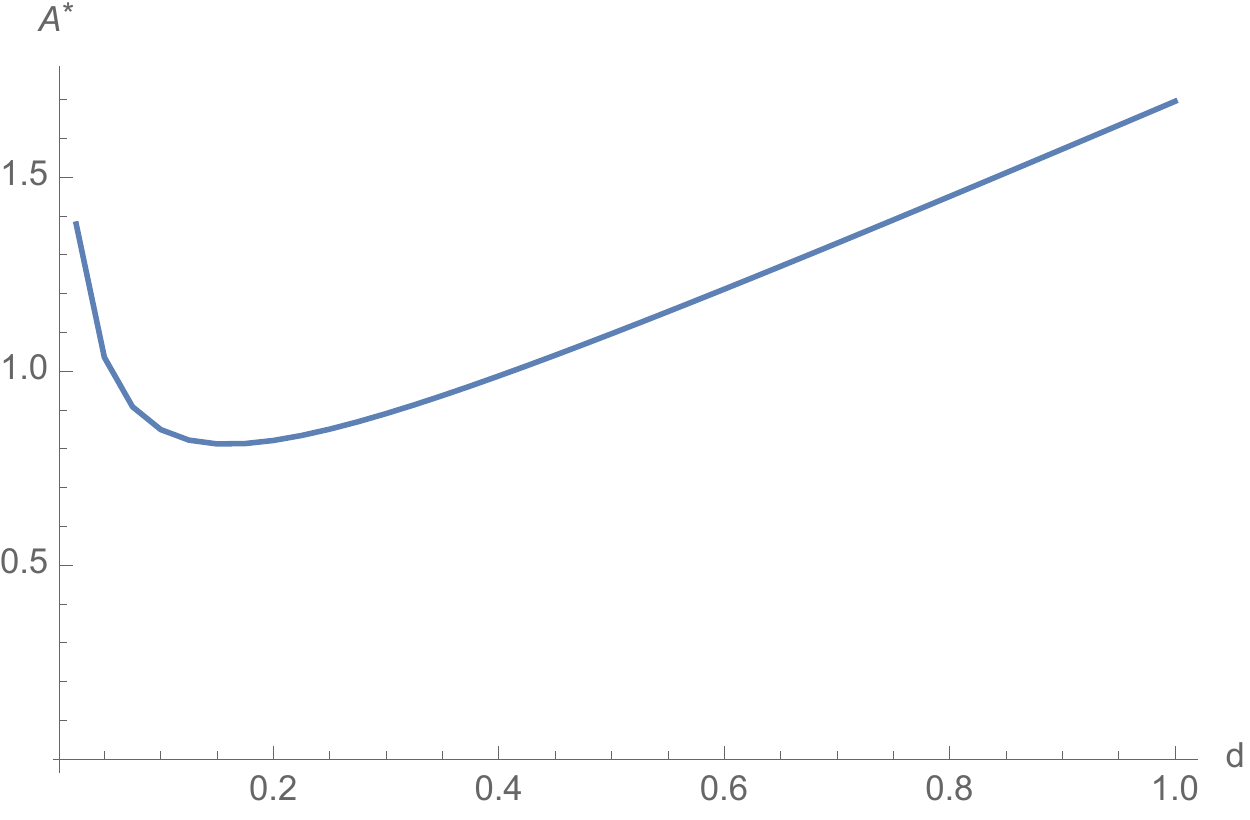}
		\caption{The unique positive root of the function \eqref{defF} for $\omega = 1, k = 1.1, c=0.1, n = 1$ and $\delta\in \{0,1\}$}
		\label{}
	\end{subfigure}
\centering
	\begin{subfigure}{.5\textwidth}
	\centering
	\includegraphics[width=.7\linewidth]{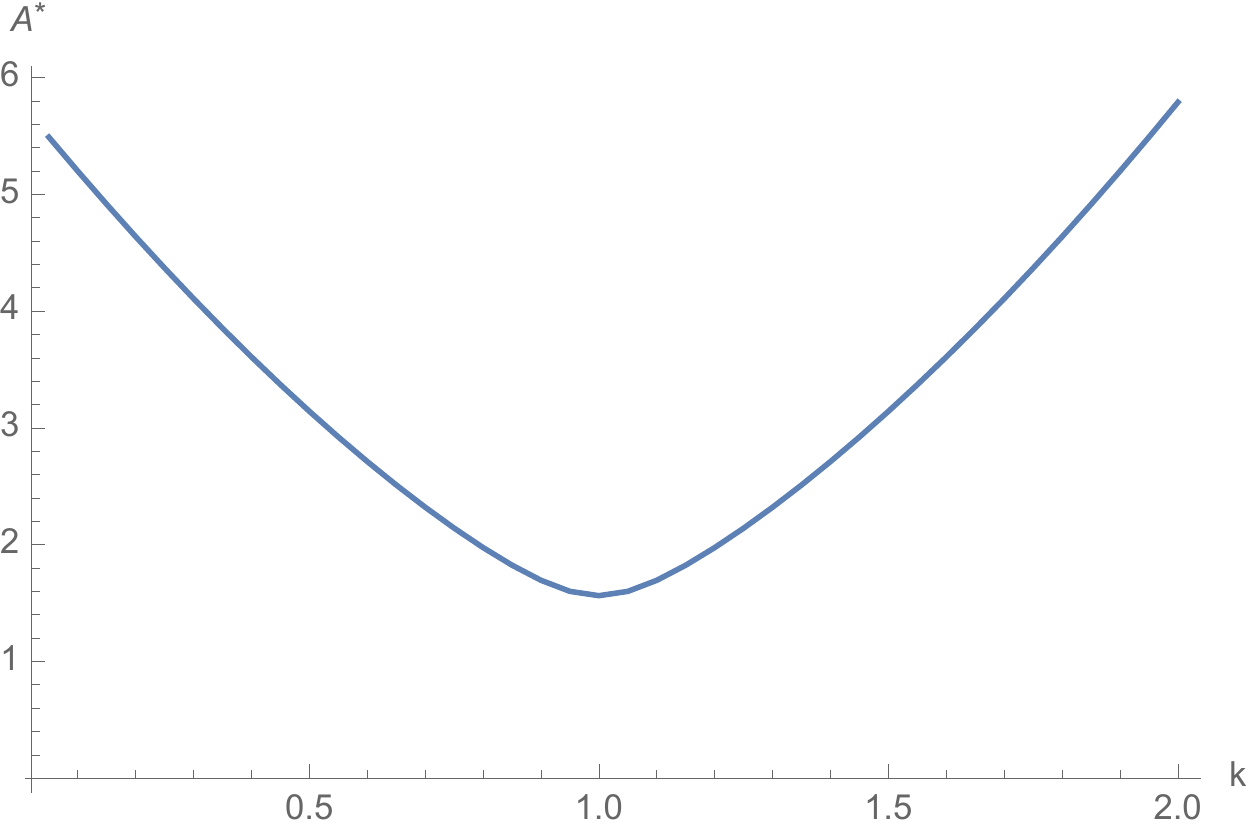}
	\caption{The unique positive root of the function \eqref{defF} for $\omega = 1, c=0.1, \delta = 1, n = 1$ and $k\in \{0,2\}$}
	\label{}
\end{subfigure}
	\caption{The behavior of the unique positive root of the equation $F(A)=0$, cf. \eqref{defF} in dependence on the damping $c$, the magnitude of the nonlinear potential $\delta$ and the linear stiffness $k$.}
	\label{Astar}
\end{figure}
\end{example}

\begin{example}\label{Duffingsoftexpl}
Consider the forced-damped Duffing oscillator with a softening cubic stiffness,
\begin{equation}\label{Duffingsoft}
\ddot{x}+c\dot{x}+kx=\delta x^3+A\sin(n\omega t),
\end{equation}
for $c>0$ the linear damping coefficient, $k>0$ the linear stiffness coefficient, $\delta>0$ the cubic stiffness coefficient, $A>0$ the forcing magnitude, $\omega>0$ the fundamental forcing frequency and $n\in\mathbb{N}$ an oscillation parameter. The nonlinear potential associated to equation \eqref{Duffingsoft} is given by
\begin{equation}
U(x)=-\delta\frac{x^4}{4},
\end{equation}
which is depicted in Figure \ref{x^2-x^4}., while Figure \ref{Dyn_Duff_soft} shows the phase portrait of the unforced and undamped system, i.e., system \eqref{Duffingsoft} with $c=A=0$. We can choose the constants in \eqref{AssUsoft} as 
\begin{equation}\label{UDuffingsoft}
u_0=\delta, \quad r=4.
\end{equation}
As our domain of validity, we choose the unit interval $I=[0,1]$, where we can bound the derivative of the potential $U$ as 
\begin{equation}
|U'(x)|=\delta|x|^3\leq \delta.
\end{equation}
This implies that we can choose $U_0=\delta$. Consequently, condition \eqref{nonexsoft} takes the form
\begin{equation}\label{condDuffsoft}
\begin{split}
\frac{\pi}{\omega}A^2>&\delta\|A\sin(n\omega t)\|_{L^1(0,T)}+\|A(k-n^2\omega^2)\sin(n\omega t)-Acn\omega\cos(n \omega t)\|_{L^{\frac{4}{3}}(0,T)}y^*,
\end{split}	
\end{equation}
where $y^*$ is the unique positive root of the polynomial
\begin{equation}
P(y)=\delta y^3-k\sqrt{T} y-A\|\sin(n\omega t)\|_{L^{\frac{4}{3}}(0,T)}.
\end{equation}
Consider again the function
\begin{equation}\label{defFsoft}
F(A):=\frac{\pi}{\omega}A^2-\delta\|A\sin(n\omega t)\|_{L^1(0,T)}-\|A(k-n^2\omega^2)\sin(n\omega t)-Acn\omega\cos(n \omega t)\|_{L^{\frac{4}{3}}(0,T)}y^*.
\end{equation}
A sign change in \eqref{defFsoft} indicates that any (potential) periodic solution leaves the domain of validity. The dependence of the unique positive root $A^*$ of \eqref{defFsoft} on different parameters is depicted in Figure \ref{Astarsoftfig}. Similar to the Duffing oscillator with hardening stiffness \eqref{Duffinghard}, the critical amplitude $A^*$ grows for larger values of $c$ and $k$, cf. Figure \ref{Astarsoftc} and Figure \ref{Astarsoftk}. On the other hand, the critical amplitude decreases for larger values of $\delta$, cf. Figure \ref{Astarsoftd}.

\begin{figure}
	\begin{subfigure}{.5\textwidth}
		\centering
		\includegraphics[width=.8\linewidth]{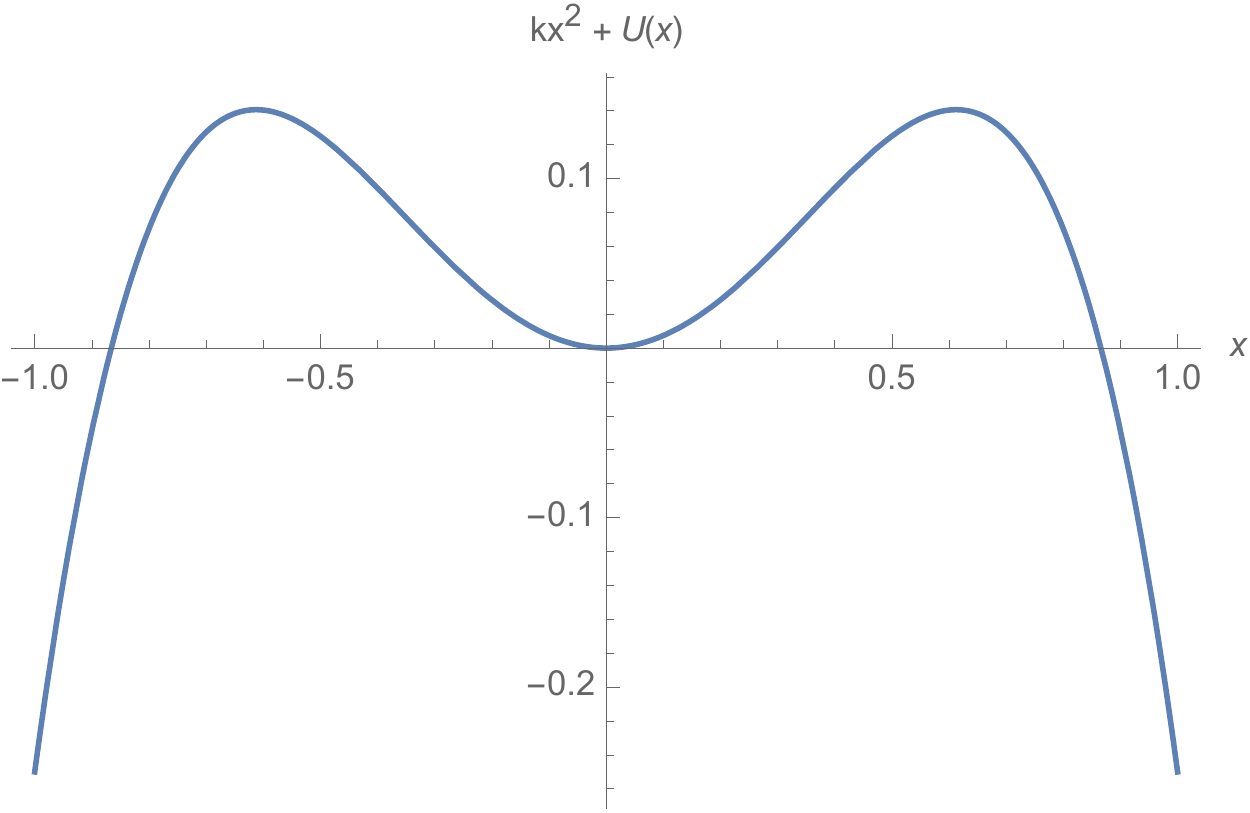}
		\caption{Linear plus nonlinear potential, $kx^2+U(x)=0.75x^2-x^4$}
		\label{x^2-x^4}
	\end{subfigure}%
	\begin{subfigure}{.5\textwidth}
		\centering
		\includegraphics[width=.7\linewidth]{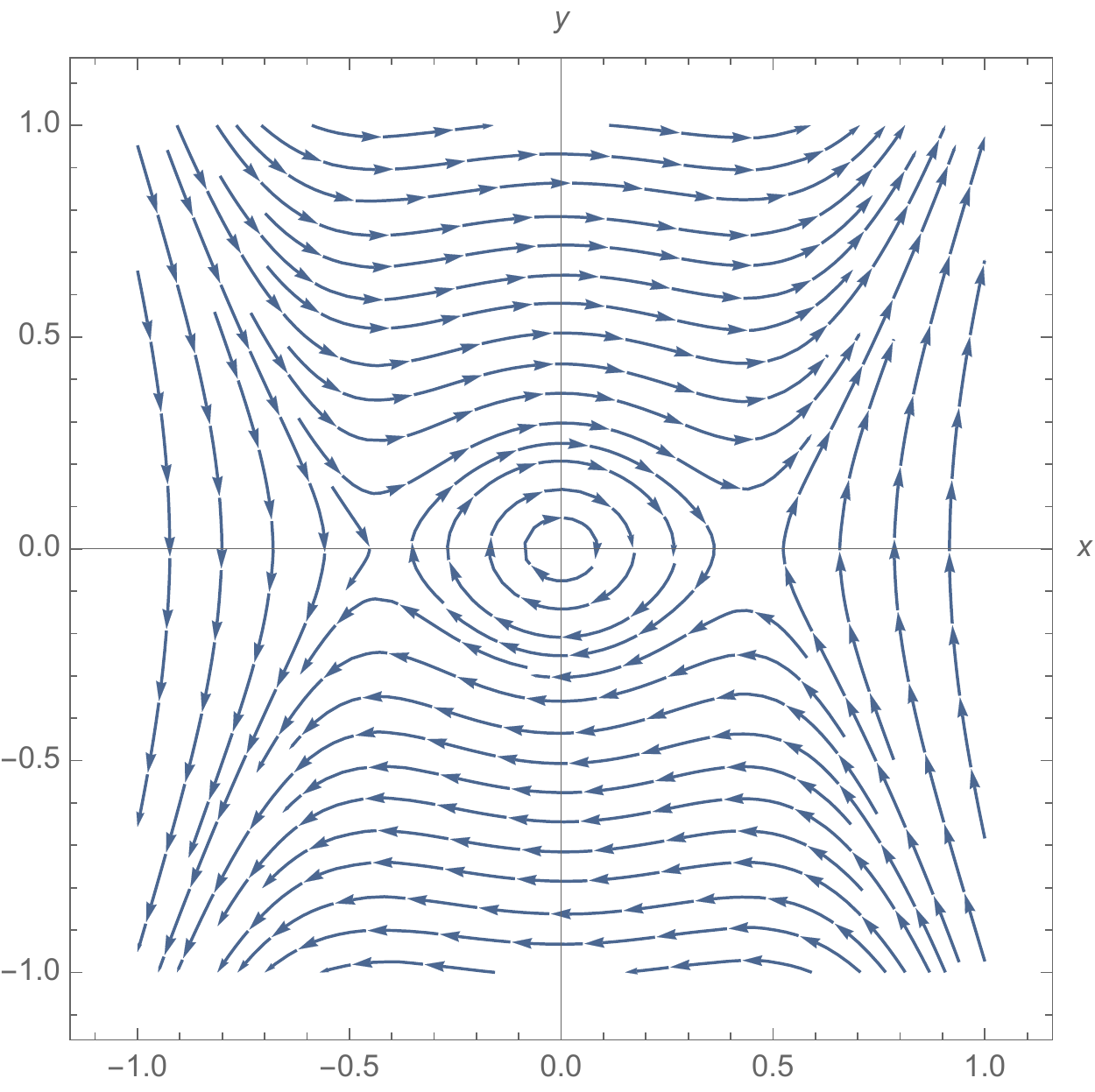}
		\caption{Dynamics of system \eqref{Duffingsoft} around the origin for $c=A=0$.}
		\label{Dyn_Duff_soft}
	\end{subfigure}
	\caption{The potential $U$ and the phase portrait of the unforced and undamped hardening Duffing oscillator.}
	\label{}
\end{figure}

	\begin{figure}
	\begin{subfigure}{.5\textwidth}
		\centering
		\includegraphics[width=.8\linewidth]{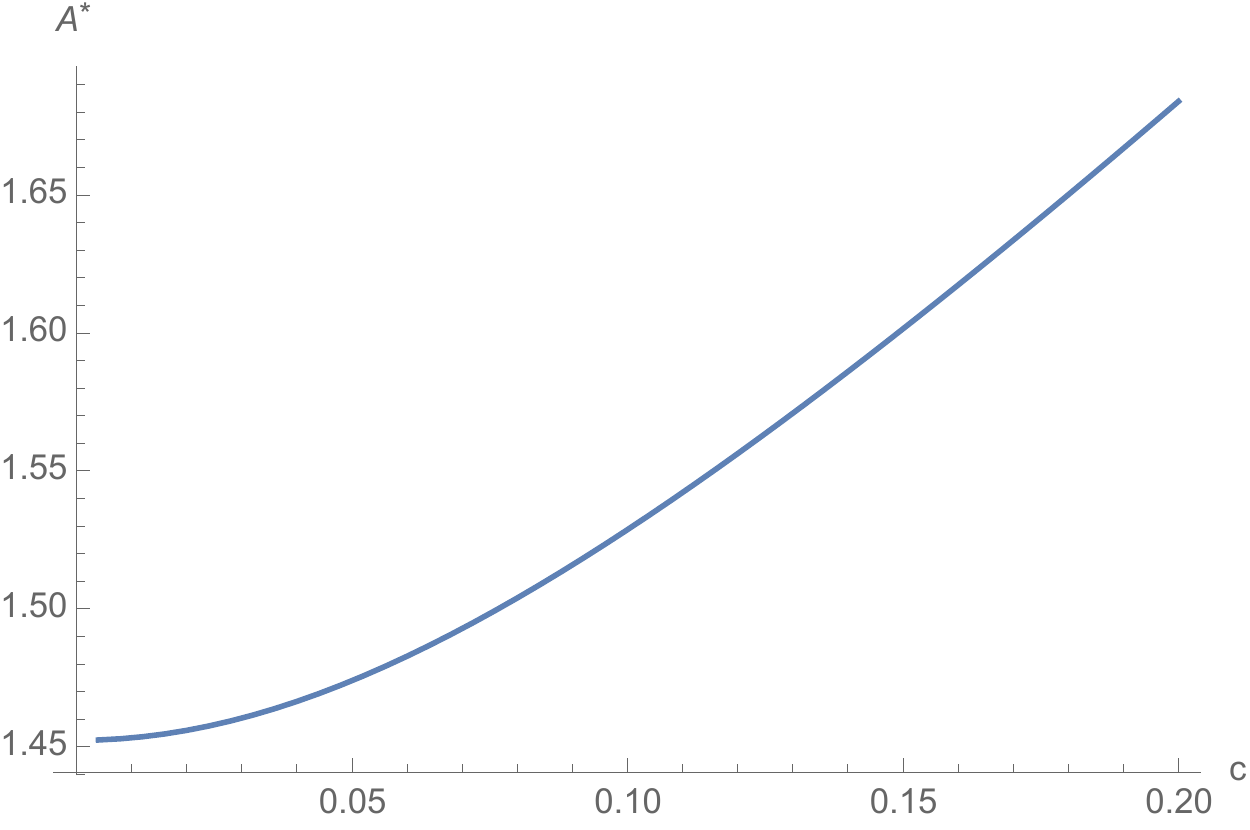}
		\caption{The unique positive root of the function \eqref{defFsoft} for $\omega = 1, k = 1.1, d = 1, n = 1$ and $c\in \{0,0.2\}$}
		\label{Astarsoftc}
	\end{subfigure}%
	\begin{subfigure}{.5\textwidth}
		\centering
		\includegraphics[width=.7\linewidth]{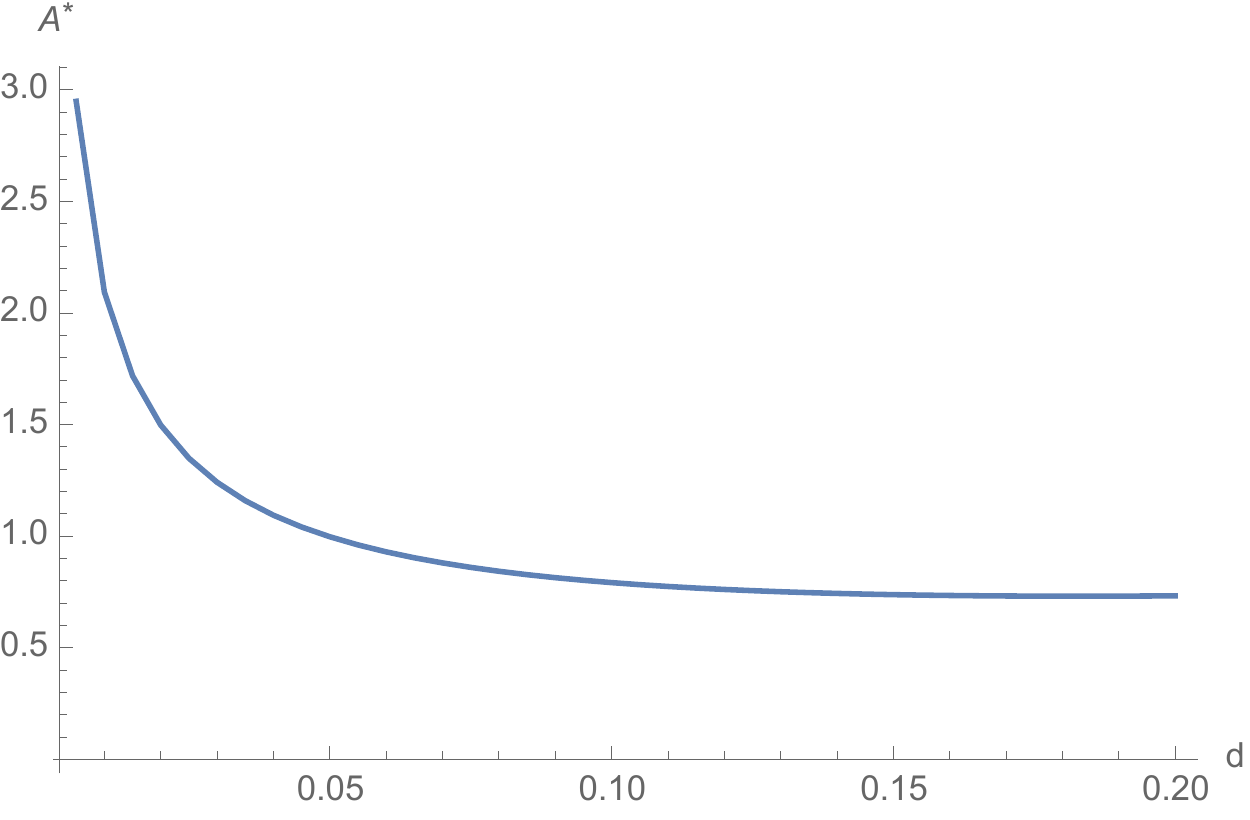}
		\caption{The unique positive root of the function \eqref{defFsoft} for $\omega = 1, k = 1.1, c=0.1, n = 1$ and $d\in \{0,1\}$}
		\label{Astarsoftd}
	\end{subfigure}
	\centering
	\begin{subfigure}{.5\textwidth}
		\centering
		\includegraphics[width=.7\linewidth]{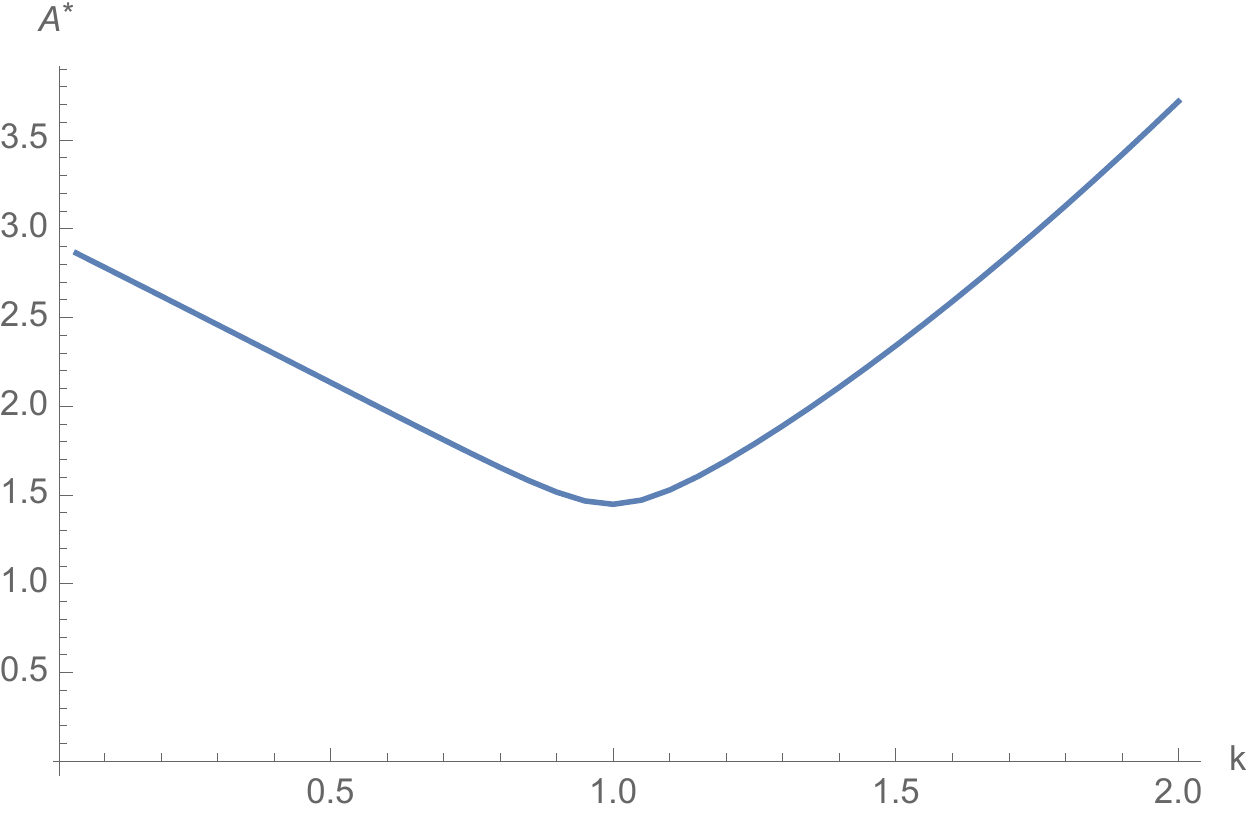}
		\caption{The unique positive root of the function \eqref{defFsoft} for $\omega = 1, c=0.1, d = 1, n = 1$ and $k\in \{0,2\}$}
		\label{Astarsoftk}
	\end{subfigure}
	\caption{The behavior of the unique positive root of the equation $F(A)=0$, cf. \eqref{defFsoft} in dependence on the damping $c$, the magnitude of the nonlinear potential $\delta$ and the linear stiffness $k$.}
	\label{Astarsoftfig}
\end{figure}
\end{example}

\section{Discussion and Further Perspectives}

We derived $L^p$-estimates for periodic solutions of forced-damped mechanical systems with nonlienarly hardening and nonlinearly softening potentials. Thanks to polynomial bounds on the nonlinear part of the potential, these estimates do not scale neutrally in the amplitude of the external forcing. Assuming an upper bound on the gradient of the potential on a bounded domain, we can show that, for sufficiently high forcing amplitude, \textit{any} periodic orbit will the given domain (at some time).\\

For higher dimensional systems, the mechanical potential is usually not perfectly symmetric with respect to its coordinates, which implies that conditions \eqref{AssUhard} and \eqref{AssUsoft} will not be satisfied in these cases. I would be interesting to extend the results of the present paper to systems with the aforementioned properties.\\
Also, the role of the frequency and the oscillation parameter - even for systems with one mechanical degree of freedom - is not completely clear. Specifically, we can ask: How does the maximum of any periodic orbit changes quantitatively with respect to the forcing amplitude?\\
The main result of the present paper can be summarized as a lower bound on the maximum of any periodic orbit (in a given domain). Can we also derive a lower bound on the minimum of a periodic orbit in a forced-damped system? This would imply a direct analogue of the Bendixson--Dulac criterion for forced-damped mechanical systems.\\

\textbf{Acknowledgments.}\\
The author would like to thank Thomas Breunung and George Haller for several useful comments and suggestions.

\section{Appendix: An estimate on the positive root of certain polynomials}
In this section, we prove an upper bound on the roots of certain polynomial equations, appearing in the estimates for nonlienarly softening potentials. They are of theoretical interest, as, for the analytical use of the estimates with nonlinearities greater than four, no explicit solution formula for the polynomial roots exists.

\begin{lemma}\label{polybound}
	Let $A$, $B$ and $C$ be positive real numbers and let $s\geq2$ be a real number. The unique positive root $y^*$ of the polynomial
	\begin{equation}
	P(y)=Ay^s-By-C,
	\end{equation}
	can be bounded from above as
	\begin{equation}\label{boundystar}
	y^*\leq \overline{y}+\sqrt{2\frac{|P(\overline{y})|}{P''(\overline{y})}},
	\end{equation}
	where
	\begin{equation}
	\overline{y}=\left(\frac{B}{sA}\right)^{\frac{1}{s-1}}.
	\end{equation}
\end{lemma}
\begin{proof}
	First, we note that the derivative of the polynomial $P$,
	\begin{equation}
	P'(y)=sAy^{s-1}-B,
	\end{equation}
	has a unique positive zero, which we denote as
	\begin{equation}\label{zerop'}
	\overline{y}=\left(\frac{B}{sA}\right)^{\frac{1}{s-1}}.
	\end{equation}
	In particular,
	\begin{equation}\label{poverliney}
	\begin{split}
	P(\overline{y})&=A\left(\frac{B}{sA}\right)^{\frac{s}{s-1}}-B\left(\frac{B}{sA}\right)^{\frac{1}{s-1}}-C=\frac{(AB)^{\frac{s}{s-1}}}{(As)^{\frac{s+1}{s-1}}}\left(s^{\frac{1}{s-1}}-s^{\frac{s}{s-1}}\right)-C\\&<0,
	\end{split}
	\end{equation}
	which follows from $s^{s-1}>1$, by assumption.\\
	Since $y\mapsto P(y)$ is convex for $y>0$,
	\begin{equation}
	P''(y)=s(s-1)Ay^{s-2}> 0,
	\end{equation}
	by the assumption $s\geq 2$, it follows that, indeed, $P$ has a unique positive solution $y^*$ with $\overline{y}\leq y^*$.\\
	By \eqref{zerop'}, we can write
	\begin{equation}\label{intp}
	P(y)=\int_{\overline{y}}^{y}\int_{\overline{y}}^{\eta}p''(\xi)\, d\xi\, d\eta + P(\overline{y}).
	\end{equation}
	Since, again by the assumption that $s\geq2$, the function $y\mapsto P''(y)$ is monotonically increasing,
	\begin{equation}
	P'''(y)=s(s-1)(s-2)y^{s-3}\geq 0,
	\end{equation}
	for $y>0$, and it follows that $\inf_{\xi\in[\overline{y},\eta]}P''(\xi)=P''(\overline{y})$, for any $\eta\geq\overline{y}$. Therefore, we can bound \eqref{intp} as
	\begin{equation}\label{lowerp}
	\begin{split}
	P(y)&=\int_{\overline{y}}^{y}\int_{\overline{y}}^{\eta}P''(\xi)\, d\xi\, d\eta + P(\overline{y})\\&\geq \int_{\overline{y}}^{y}\int_{\overline{y}}^{\eta}P''(\overline{y})\, d\xi\, d\eta + P(\overline{y})=P''(\overline{y})\int_{\overline{y}}^{y}(\eta-\overline{y})\, d\eta + P(\overline{y})\\&=\frac{P''(\overline{y})}{2}y^2-\overline{y}P''(\overline{y})y+\frac{P''(\overline{y})\overline{y}^2}{2}+P(\overline{y}),
	\end{split}
	\end{equation}
	for $y\geq\overline{y}$.\\
	In particular, the unique positive zero of $P$ can be bounded from above by the positive root of the right-hand side in \eqref{lowerp}, i.e., by
	\begin{equation}
	y^{+}=\overline{y}+\sqrt{2\frac{|P(\overline{y})|}{P''(\overline{y})}},
	\end{equation}
	where we have used \eqref{poverliney}.
	This proves the claim.
\end{proof}

\begin{remark}The quadratic function $y\mapsto \frac{P''(\overline{y})}{2}y^2-\overline{y}p''(\overline{y})y+\frac{P''(\overline{y})\overline{y}^2}{2}+P(\overline{y})$ in \eqref{lowerp} defines a parabola with crest at $\overline{y}$. Since the global minimum of $P$ is attained at $\overline{y}$ as well and since the growth rate of $P$ is greater or equal than the growth rate of the parabola by the assumption $p\geq 2$, we can, indeed, bound the zero of $P$ from above by the positive zero of the parabola, c.f. Figure \ref{fig1}.\\We note that - thanks to the special structure of the polynomial, the bound \eqref{boundystar} improves classical, general a-priori bounds on the zeros of polynomials. We compare \eqref{boundystar} e.g. with the Lagrange bound \cite{lagrange1806traite},\begin{equation}\label{Lagrange}y^*\leq \max\left\{1,\sum_{k=0}^{n-1}{\left|\frac{P_k}{P_n}\right|}\right\},\end{equation}for the polynomial $P(y)=\sum_{k=1}^{n}P_ky^k$. For the polynomial $P(y)=y^5-y-1$, for which $y^*=1.1673$, we find that \eqref{Lagrange} predicts $y^*\leq 2$, while the parabolic bound  \eqref{boundystar} predicts $y^*\leq 1.38516$.\end{remark}\begin{figure}[ht]	\centering	\includegraphics[width=0.8\textwidth]{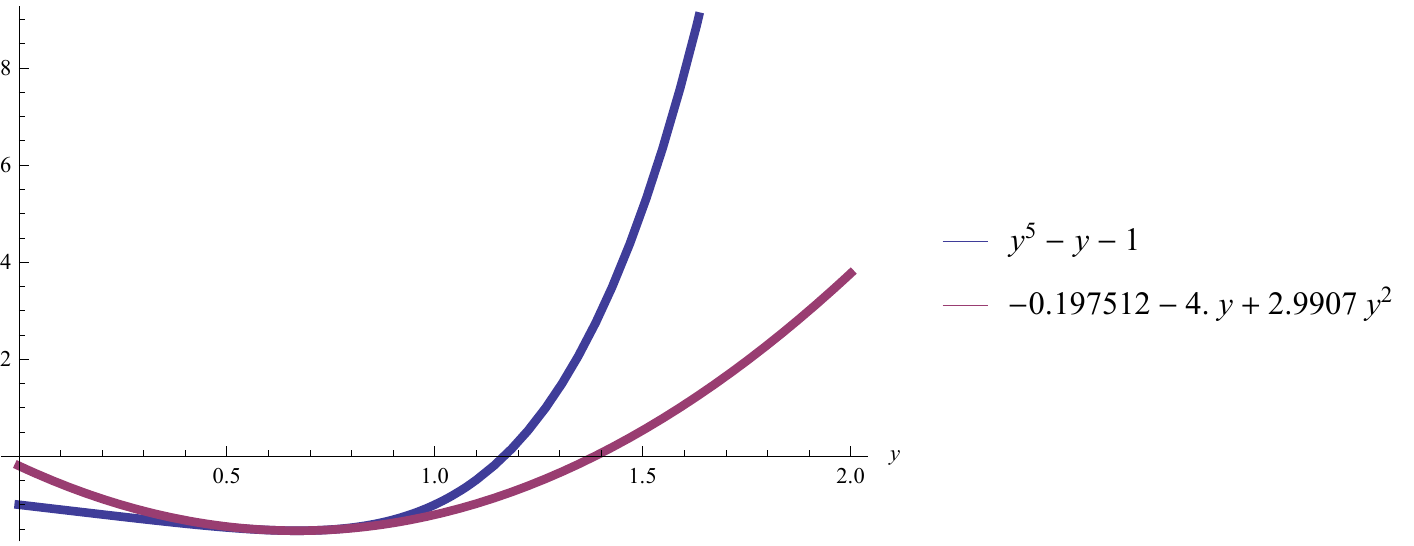}	\caption{The polynomial functions $P(y)=y^5-y-1$ and the parabola \eqref{lowerp} with $\overline{y}=0.66874$. The positive zero of the parabola is attained at $1.38516$, while the unique positive zero of $P$ is attained at $1.1673$.}	\label{fig1}\end{figure}

\bibliographystyle{abbrv}
\bibliography{/Users/floriankogelbauer/Dropbox/Bibs/DynamicalSystems.bib}
\end{document}